\documentclass[12pt]{amsart}

\usepackage[english]{babel}
\usepackage[letterpaper,top=2cm,bottom=2cm,left=3cm,right=3cm,marginparwidth=1.75cm]{geometry}

\usepackage{amsmath}
\usepackage{amsthm}
\usepackage{amssymb}
\usepackage[draft]{graphicx}
\usepackage{mathtools}
\usepackage{comment}
\usepackage{tikz-cd} 

\theoremstyle{plain}
\newtheorem{theorem}{Theorem}[section]
\newtheorem{lemma}[theorem]{Lemma}
\newtheorem{proposition}[theorem]{Proposition}
\newtheorem{corollary}[theorem]{Corollary}

\theoremstyle{definition}
\newtheorem{definition}[theorem]{Definition}
\newtheorem{remark}[theorem]{Remark}

\usepackage{color}
\definecolor{e-mail}{rgb}{0,.40,.80}
\definecolor{reference}{rgb}{.20,.60,.22}
\definecolor{mrnumber}{rgb}{.80,.40,0}
\definecolor{citation}{rgb}{0,.40,.80}
\usepackage[breaklinks=true,colorlinks=true,linkcolor=reference, citecolor=citation, urlcolor=e-mail]{hyperref}
\newcommand{\frp}{\mathfrak{p}}

\newcommand{\cyc}{\mathrm{cyc}}
\newcommand{\E}{\mathrm{ell}}
\newcommand{\ab}{\mathrm{ab}}

\newcommand{\cw}{\mathrm{CW}}
\newcommand{\intlog}{\mathcal{L}}

\DeclareMathOperator{\tr}{Tr}
\DeclareMathOperator{\nr}{Nr}
\DeclareMathOperator{\cnr}{\mathcal{N}}
\DeclareMathOperator{\ctr}{\overline{\mathcal{S}}}
\DeclareMathOperator{\gal}{Gal}
\DeclareMathOperator{\rec}{rec}
\DeclareMathOperator{\res}{res}

\title[]{CM Elliptic Analogue of the Coleman--Ihara Formula}
\date{}

\author[]{Yuki Goto$^{\dagger\diamond}$}
\address{%
$\dagger$ Graduate School of Science and Technology,
Keio University, 3-14-1 Hiyoshi,
Kouhoku-ku, Yokohama 223-8522, Japan.}
\address{%
$\diamond$ RIKEN Center for Advanced Intelligence Project (AIP),
1-4-1 Nihonbashi, Chuo-ku,
Tokyo 103-0027, Japan.}
\email{yuki0831@keio.jp}
\author[]{Shun Ishii$^{\ddagger}$}
\address{%
$\ddagger$ Research Institute for Mathematical Sciences (RIMS),
Kyoto University, Kyoto 606-8502, Japan.}

\begin{document}

\begin{abstract}
The Soul\'{e} characters are Kummer characters constructed from
cyclotomic units and also arising from the Galois action on the pro-$p$
geometric fundamental group of $\mathbb{P}^1\setminus\{0,1,\infty\}$.
The Coleman--Ihara formula describes the restrictions of these characters
to the local Galois group at $p$ as products of the Coates--Wiles
homomorphisms and Kubota--Leopoldt $p$-adic $L$-values at positive odd
integers.
In this paper, we prove an elliptic analogue of this formula for elliptic
curves with complex multiplication by the ring of integers of an imaginary
quadratic field $K$.
We consider certain Kummer characters associated with elliptic units,
which we call elliptic Soul\'{e} characters.
These characters are elliptic counterparts of the classical Soul\'{e}
characters and also arise from the Galois action on the pro-$p$ geometric
fundamental groups of once-punctured elliptic curves.
For suitable pairs of positive integers $(i,j)$, we give an explicit
description of their restrictions to the relevant local Galois groups as
products of the two-variable Coates--Wiles homomorphisms introduced by
Yager and the two-variable $p$-adic $L$-values of $K$.
\end{abstract}

\maketitle

\begin{center}
\tableofcontents
\end{center}

\section{Introduction}
Let $p$ be an odd prime. For any field $F$ of characteristic zero, write
$G_F:=\gal(\overline{F}/F)$ for its absolute Galois group. 
For each $n\geq 1$, let $\mu_{p^n}$ denote the group of $p^n$-th roots of unity, 
and fix a compatible system $\{\zeta_{p^n}\}_{n\geq 1}$ of primitive $p^n$-th roots of unity.
This induces an identification $\mathbb{Z}_p \xrightarrow{\sim} 
\mathbb{Z}_p(1):=\varprojlim_n \mu_{p^n}$ sending $1$ to $(\zeta_{p^n})_n$.
The Soul\'e characters are Kummer characters arising from cyclotomic units.
More precisely, for an odd integer $m\geq 3$, the $m$-th Soul\'e character 
$\chi_m\colon G_{\mathbb{Q}(\mu_{p^\infty})}\rightarrow \mathbb{Z}_p(m) := \mathbb{Z}_p(1)^{\otimes m}$
is defined by the following Kummer property: for $\sigma\in G_{\mathbb{Q}(\mu_{p^\infty})}$ and $n\geq 1$,
\[
\zeta_{p^n}^{\chi_m(\sigma)} = \left(\prod_{\substack{1 \le a < p^n \\ p \nmid a}} 
(1-\zeta_{p^n}^a)^{\frac{a^{m-1}}{p^n}} \right)^{\sigma-1}, 
\]
where $\chi_m(\sigma)$ is regarded as an element of $\mathbb{Z}_p$ via the
identification $\mathbb{Z}_p(m)\cong\mathbb{Z}_p$ induced by the above choice. 
The character $\chi_m$ originates in Soul\'{e}'s work on higher $p$-adic
regulators \cite{Sou81} and is obtained by restricting a cohomology class in
$H^1(G_{\mathbb{Q}},\mathbb{Z}_p(m))$, constructed from cyclotomic units
via Kummer theory, to $G_{\mathbb{Q}(\mu_{p^\infty})}$.
It also admits a geometric interpretation through Ihara's
study of the Galois action on the pro-$p$ geometric \'{e}tale fundamental group
of $\mathbb{P}^1\setminus\{0,1,\infty\}$ \cite{Iha86}.

The local behavior of the Soul\'e character $\chi_m$ at $p$ is described 
by the Coleman--Ihara formula.
For each $n \geq 1$, let $U_n$ be the group of principal units of
$\mathbb{Q}_p(\mu_{p^n})$, and set $\mathbf{U}_\infty:=\varprojlim_n U_n$
with respect to the norm maps. Local class field theory gives the
reciprocity map $\rec\colon \mathbf{U}_\infty\to G_{\mathbb{Q}_p(\mu_{p^\infty})}^{\ab}$.
Let $\phi_m^\cw\colon \mathbf{U}_\infty\to \mathbb{Z}_p(m)$ be the
$m$-th Coates--Wiles homomorphism, and let $L_p(s,\omega^{1-m})$ be the
Kubota--Leopoldt $p$-adic $L$-function attached to $\omega^{1-m}$, where
$\omega$ denotes the Teichm\"{u}ller character.
Fix an embedding $\overline{\mathbb{Q}}\hookrightarrow \overline{\mathbb{Q}}_p$,
and identify the corresponding decomposition subgroup of $G_{\mathbb{Q}(\mu_{p^\infty})}$ 
with $G_{\mathbb{Q}_p(\mu_{p^\infty})}$. With this convention, the classical
Coleman--Ihara formula asserts that
\[
\frac{\chi_m(\rec(\boldsymbol{\epsilon}))}{p^{m-1}-1} = L_p(m,\omega^{1-m})\phi_m^{\mathrm{CW}}(\boldsymbol{\epsilon}) \qquad (\boldsymbol{\epsilon}\in \mathbf{U}_\infty).
\]
Thus, it describes the restriction of the Soul\'{e} character $\chi_m$ to the
local Galois group at $p$ in terms of the Coates--Wiles homomorphism and a
Kubota--Leopoldt $p$-adic $L$-value.

Historically, the Coleman--Ihara formula arose from Ihara's study of his 
universal power series for Jacobi sums \cite{Iha86}.
From the coefficients of the logarithm of this power series,
Ihara obtained certain characters and local formulae relating them to
Coates--Wiles homomorphisms and Kubota--Leopoldt $p$-adic $L$-values.
Combined with Coleman's identification of these characters with the restrictions of 
Soul\'{e} characters, these formulae yield the comparison now known as the Coleman--Ihara formula,
whose proof relies on Coleman's explicit reciprocity law for Hilbert norm residue symbols \cite{Col81,Iha86}.
In subsequent work, Anderson, Coleman, and Ihara--Kaneko--Yukinari
developed explicit global formulae in this setting
\cite{And89,Col89,IKY87}.
Although these global formulae reveal deep connections with Iwasawa theory and the arithmetic of
cyclotomic fields, the Coleman--Ihara formula itself is essentially local:
it describes the restriction of a Soul\'{e} character to the image of the local reciprocity map.
This local viewpoint has led to several generalizations.
Nakamura, Sakugawa, and Wojtkowiak proved polylogarithmic analogues, 
replacing Soul\'{e} characters and Kubota--Leopoldt $p$-adic $L$-values by 
$p$-adic polylogarithmic characters and Coleman $p$-adic polylogarithms, respectively \cite{NSW17a,NSW17b}.
Burns and Sano formulated a higher-rank analogue using Euler systems and
generalized Stark elements, which recovers the classical formula over
$\mathbb{Q}$ and gives a conjectural extension to totally real fields
\cite{BS25}.

In this paper, we prove a two-variable CM elliptic analogue of the Coleman--Ihara formula.
Let $K$ be an imaginary quadratic field of class number one with ring of integers $O_K$, 
and put $w_K:=\#O_K^\times$.
We fix an elliptic curve $E$ with complex multiplication by $O_K$. 
Let $p \ge 5$ be a prime that splits in $K$ as $(p)=\frp\bar{\frp}$.
We choose a suitably normalized generator $\pi$ of $\frp$, 
and write $\bar{\pi}$ for its complex conjugate.
The analogues of Soul\'{e} characters in this setting are the elliptic 
Soul\'{e} characters $\chi_{(i,j)}^\E$ (cf.~Definition~\ref{def:elliptic_soule}), 
originating in Nakamura's work on Galois 
representations associated with once-punctured elliptic curves \cite{Na95}.
A Coleman--Ihara-type formula for the characters
$\chi_{(m,m)}^\E$ along the cyclotomic line was previously obtained in
\cite{IG26}. In the present paper, we generalize this to the
two-variable setting, namely to $\chi_{(i,j)}^\E$ for positive integers
$i,j$ satisfying $i\equiv j\pmod{w_K}$.
The formula below involves the two-variable $p$-adic $L$-function
$L_p(s_1,s_2,\omega^{(i,j)})$ of $K$ 
(cf.~Definition~\ref{def:two-variable_p-adic_L}), as studied by de Shalit
\cite{dS87}, where $\omega^{(i,j)}$ is a certain finite character.
We write $\mathbf{U}_\infty^\E$ for the group of norm-compatible systems
of principal units associated with the tower of $p$-power division fields
of $E$.
The logarithmic derivatives of the corresponding Coleman power series give
rise to the two-variable Coates--Wiles homomorphisms $\phi_{(i,j)}^\cw$ 
(cf.~Definition~\ref{def:coates-wiles}), introduced by Yager \cite{Yag82}.
We also let $\Omega_p$ denote a $p$-adic period associated with $E$, and write
$\rec_\infty$ for the reciprocity map on $\mathbf{U}_\infty^\E$
(cf.~\eqref{eq:reciprocity-map}).
The main result of this paper is the following.

\begin{theorem}[{Corollary~\ref{cor:CM_elliptic_Coleman_Ihara}}]
\label{thm:intro_CM_elliptic_Coleman_Ihara}
For each pair of positive integers $(i,j) \neq (1,1)$ such that
$i\equiv j\pmod{w_K}$, the following identity holds:
\[
\begin{aligned}
    \frac{\chi_{(i,j)}^\E(\rec_\infty(\boldsymbol{\beta}))}
    {1-\pi^{j-1}\bar{\pi}^{i-1}}
        &= 12w_K
        \left(
        1-\frac{1}{1-\pi^{i-1}\bar{\pi}^{j-1}}
        -\frac{1}{1-\pi^{j-1}\bar{\pi}^{i-1}}
        \right) \\
        &\phantom{={}}
        \cdot
        \frac{L_p(1-i,1-j,\omega^{(i,j)-1})}
        {\Omega_p^{j-i}}
        \,\phi_{(i,j)}^\cw(\boldsymbol{\beta}) \qquad (\boldsymbol{\beta}\in \mathbf{U}_\infty^\E).
\end{aligned}
\]
Here, we regard $\chi_{(i,j)}^\E$ as a character on the product of the
local Galois groups at the primes above $\frp$ (cf.~\eqref{eq:embedding_to_global}).
\end{theorem}

While preparing this manuscript, we learned that Sano and Venjakob
had independently obtained the same result  \cite{SV26}.

We prove the theorem by adapting the local method of Nakamura, Sakugawa, and
Wojtkowiak \cite{NSW17a} to the CM elliptic setting. Namely, we compare two
expressions for a Hilbert norm residue symbol. The first arises from the
Kummer property of elliptic units and yields the elliptic Soul\'e character.
The second is obtained from Coleman's explicit reciprocity law and is
expressed in terms of Coleman power series. To identify the resulting
expression with the product of the two-variable $p$-adic
$L$-function and the two-variable Coates--Wiles homomorphism, we compare the
Coleman power series arising on the cyclotomic and elliptic sides via an
isomorphism between the multiplicative formal group and the formal group of
$E$.

The paper is organized as follows.
In Section~2, we fix the notation used throughout the paper.
In Section~3, we introduce the elliptic Soul\'e characters and review their
basic properties.
Section~4 recalls the Coleman-theoretic ingredients needed for the proof:
Coleman power series in the cyclotomic and elliptic settings, Coleman's
explicit reciprocity law, and the two-variable Coates--Wiles homomorphisms 
and the two-variable $p$-adic $L$-function.
Finally, in Section~5, we combine these ingredients to prove the main
theorem and verify its compatibility with the formula obtained in
\cite{IG26}.
\section{Notation}\label{sec:notation}
In this section, we fix the notation used throughout this paper.
\begin{itemize}
    \item Let $K$ be an imaginary quadratic field of class number one, and let $O_K$ denote its ring of integers. 
    Write $d_K$ for the absolute value of the discriminant of $K$ and $w_K = \#O_K^\times$ for the number of roots of unity in $K$.
    We fix an algebraic closure $\overline{K}$ of $K$ and an embedding $\iota_\infty: \overline{K}\hookrightarrow\mathbb{C}$.

    \item Let $E$ be an elliptic curve defined over $K$ with complex multiplication by $O_K$.
    We fix a Weierstrass equation $y^2=4x^3-g_2x-g_3$ with $g_2,g_3\in K$ for $E$ and 
    identify $\mathbb{C}/L$ with $E(\mathbb{C})$ via the Weierstrass uniformization
    \[
    \xi(\;\cdot\;;L)\colon \mathbb{C}/L\xrightarrow{\sim}E(\mathbb{C}), \; z\mapsto \bigl(\wp(z;L):\wp'(z;L):1\bigr),
    \]
    where $\wp(z;L)$ is the Weierstrass $\wp$-function associated with the period lattice $L$.
    We also fix a period $\Omega\in\mathbb{C}^\times$ such that $L=\Omega O_K$. 
    For an ideal $\mathfrak{a}$ of $O_K$, we denote by $E[\mathfrak{a}]$ the $\mathfrak{a}$-torsion subgroup of $E$.

    \item We fix a prime $p\ge 5$ that splits in $K$ as $(p)=\frp\bar{\frp}$.
    We assume that $E$ has good reduction at $\frp$.
    We choose $\pi$ to be the generator of $\frp$ characterized by the convention
    \[
    (\frp,\,K(E[\bar{\frp}^\infty])/K)(u)=[\pi]_E(u)
    \qquad \text{for all } u \in E[\bar{\frp}^\infty]
    := \bigcup_{m\ge0} E[\bar{\frp}^m],
    \]
    where $(\frp,\,K(E[\bar{\frp}^\infty])/K)$ denotes the Artin symbol at $\frp$.
    Let $\bar{\pi}$ be the complex conjugate of $\pi$.
    Let $K_{\frp}$ and $K_{\bar{\frp}}$ denote the completions of $K$ at
    $\frp$ and $\bar{\frp}$, respectively, and write $O_{\frp}$ and
    $O_{\bar{\frp}}$ for their rings of integers.
    Since $p$ splits in $K$, we have
    $O_{\frp}\cong O_{\bar{\frp}}\cong \mathbb{Z}_p$.

    \item For each $n\ge 0$, the decomposition
    $O_K/p^nO_K\simeq O_K/\frp^n\oplus O_K/\bar{\frp}^n$ induces an isomorphism
    $
    E[p^n]\xrightarrow{\sim} E[\frp^n]\oplus E[\bar{\frp}^n].
    $
    We fix $\omega_{\frp,n}\in E[\frp^n]$ and
    $\omega_{\bar{\frp},n}\in E[\bar{\frp}^n]$ such that the image of
    $\xi(\Omega/p^n;L) \in E[p^n]$ under this isomorphism is
    $(-\omega_{\frp,n},\omega_{\bar{\frp},n})$.
    Let $\mu_{p^n}$ denote the group of all $p^n$-th roots of unity, and let
    $\zeta_{p^n}\in\mu_{p^n}$ be the image of
    $(\omega_{\bar{\frp},n},\omega_{\frp,n})$ under the Weil pairing
    $E[p^n]\times E[p^n]\to\mu_{p^n}$.
    We fix an embedding
    $\iota_p\colon \overline{K}\hookrightarrow\mathbb{C}_p$
    corresponding to a place of $\overline{K}$ above $\frp$, together with a
    field isomorphism
    $\mathbb{C}\xrightarrow{\sim}\mathbb{C}_p$
    whose composition with $\iota_\infty$ is $\iota_p$.
    When there is no ambiguity, we identify $\mathbb{C}$ with
    $\mathbb{C}_p$ via this isomorphism and omit $\iota_\infty$ and
    $\iota_p$ from the notation.
    Then we have $\zeta_{p^n}^{\sqrt{-d_K}} = \exp(-2\boldsymbol{\pi}\boldsymbol{i}/p^n)$ 
    (see \cite[II.6.2]{dS87}), where $\boldsymbol{\pi}$ denotes 
    the usual circle constant and $\boldsymbol{i}=\sqrt{-1}$.
    
    \item For integers $m, n \ge 0$, we define
    \[
    F_m:=K(E[\bar{\frp}^m]), \qquad K_{m,n}:=F_m(E[\frp^n]) = K(E[\frp^n\bar{\frp}^m]).
    \]
    We note that $F_m(\mu_{p^n})=K_{m,n}$ whenever $n\le m$. 
    Set $K(E[p^\infty]) := \bigcup_{m,n \ge 0} K_{m,n}$ and $G := \gal(K(E[p^\infty])/K)$. 
    For a nonzero integral ideal $\mathfrak{a}$ of $O_K$, let $K(\mathfrak{a})$ denote 
    the ray class field of $K$ modulo $\mathfrak{a}$.
    We also let $K(p^\infty):=\bigcup_{m,n\ge 0}K(\frp^n\bar{\frp}^m)$ denote the maximal 
    abelian extension of $K$ unramified outside $p$, and write $\mathcal{G}:=\gal(K(p^\infty)/K)$.

    \item We define the following characters on $G$:
    \begin{itemize}
        \item Let $\chi_{\cyc}\colon G\to \mathbb{Z}_p^\times$ be the $p$-adic cyclotomic character, characterized by
        \[
        \sigma(\zeta_{p^n})=\zeta_{p^n}^{\chi_{\cyc}(\sigma)} \qquad (n\ge 1).
        \]
        \item Let $\kappa_1\colon G\to O_{\frp}^\times\cong \mathbb{Z}_p^\times$ and $\kappa_2\colon G\to O_{\bar{\frp}}^\times\cong \mathbb{Z}_p^\times$ be the characters describing the action of $G$ on $T_{\frp}E$ and
        $T_{\bar{\frp}}E$, characterized respectively by
        \[
        \sigma(\omega_{\frp,n})=[\kappa_1(\sigma)]_E(\omega_{\frp,n}),
        \quad
        \sigma(\omega_{\bar{\frp},n})=[\kappa_2(\sigma)]_E(\omega_{\bar{\frp},n}) \quad (n\ge 1).
        \]
    \end{itemize}
    By the Galois equivariance of the Weil pairing, we have $\kappa_1\kappa_2=\chi_{\cyc}$.
    We fix a lift $\phi\in G$ of the Artin symbol
    $(\frp,K(E[\bar{\frp}^{\infty}])/K)$ such that $\kappa_1(\phi)=1$. 
    Then $\kappa_2(\phi)=\pi$. 

    \item We introduce the formal groups used below.
    \begin{itemize}
        \item Let $\widehat{\mathbb{G}}_m$ be the formal multiplicative group over $\mathbb{Z}_p$ with parameter $S$.
        For $a\in \mathbb{Z}_p$, we define $[a]_{\widehat{\mathbb{G}}_m}(S)=(1+S)^a-1$.

        \item Let $\widehat{E}$ be the formal group of $E$ over $O_{\frp}$, with parameter $T=-2x/y$.
        Since $p$ splits in $K$ and $E$ has complex multiplication by $O_K$, the formal group $\widehat{E}$ has height $1$ over $O_{\frp}$ and carries its natural $O_{\frp}$-module structure.
        For $a\in O_{\frp}\cong \mathbb{Z}_p$, write $[a]_{\widehat{E}}$ for the corresponding endomorphism.
        For each integer $n \ge 0$, we let $\omega_n := T(\omega_{\frp, n})$.
    \end{itemize}

    \item For each integer $m\ge 0$, we define the following semilocal $K_{\frp}$-algebra:
    \[
    \mathcal{F}_m :=F_m\otimes_K K_{\frp} \cong \prod_{v\mid\frp}F_{m,v}. 
    \]
    Here and below, the products are indexed by the places $v$ of $F_m$
    above $\frp$, and $F_{m,v}$ denotes the completion of $F_m$ at $v$.
    Since the extensions $F_m(\mu_{p^n})/F_m$ and $K_{m,n}/F_m$ are
    totally ramified at every such $v$, the place $v$ extends uniquely to
    each of these fields; the corresponding completions are denoted by
    $F_{m,v}(\mu_{p^n})$ and $K_{m,n,v}$, respectively.
    We define a subring $\mathcal{O}_m$ of $\mathcal{F}_m$ by
    \[
    \mathcal{O}_m := \prod_{v\mid\frp} O_{F_{m,v}},
    \]
    Let $\widehat{\mathcal{O}}$ denote the completion of
    $\mathcal{O}_\infty:=\bigcup_{m\ge0}\mathcal{O}_m$. 
    The action of $G$ on $\mathcal{O}_\infty$ extends continuously to
    $\widehat{\mathcal{O}}$.
    The residue fields of $F_{m,v}(\mu_{p^n})$ and $K_{m,n,v}$ are
    naturally identified with that of $F_{m,v}$.
    Denote this common residue field by $k_{m,v}$ and set
    $k_m:=\prod_{v\mid\frp}k_{m,v}$.
    Then we have decompositions
    \[
    \prod_{v\mid\frp} O_{F_{m,v}(\mu_{p^n})}^\times
    \cong k_m^{\times}\times U_{m,n}^{\cyc},
    \qquad
    \prod_{v\mid\frp} O_{K_{m,n,v}}^\times
    \cong k_m^{\times}\times U_{m,n}^{\E},
    \]
    where $U_{m,n}^{\cyc}$ and $U_{m,n}^{\E}$ are the corresponding groups
    of principal units.
    In particular, $k_m$ is independent of $n$, and $k_m^\times$ is a
    finite group of order prime to $p$.
    For $x\in\prod_{v\mid\frp} O_{F_{m,v}(\mu_{p^n})}^\times$ or
    $x\in\prod_{v\mid\frp} O_{K_{m,n,v}}^\times$, write $\langle x\rangle$ 
    for its projection onto the second factor.
    Taking inverse limits with respect to the norm maps, define
    \[
    \mathbf{U}_m^{\cyc}:=\varprojlim_n U_{m,n}^{\cyc},
    \qquad
    \mathbf{U}_m^{\E}:=\varprojlim_n U_{m,n}^{\E},
    \]
    and
    \[
    \mathbf{U}_\infty^{\cyc}:=\varprojlim_m\mathbf{U}_m^{\cyc},
    \qquad
    \mathbf{U}_\infty^{\E}:=\varprojlim_m\mathbf{U}_m^{\E}.
    \]
    Since $F_m(\mu_{p^n})=K_{m,n}$ whenever $n\le m$, the identifications
    $U_{m,n}^{\E}=U_{m,n}^{\cyc}$ for $n\le m$ are compatible with the norm maps. 
    As the pairs $(m,n)$ satisfying
    $n\le m$ form a cofinal subsystem, they induce a canonical isomorphism
    $\mathbf{U}_\infty^{\E} \xrightarrow{\sim} \mathbf{U}_\infty^{\cyc}$.
    In particular, if $\boldsymbol{\beta}=(\beta_{m,n})_{m,n}\in\mathbf{U}_\infty^{\E}$
    corresponds to $\boldsymbol{\epsilon}=(\epsilon_{m,n})_{m,n} \in\mathbf{U}_\infty^{\cyc}$,
    then $\beta_{m,n}=\epsilon_{m,n}$ for $n\le m$.
    The fixed embedding $\iota_p$ induces a map from
    $\widehat{\mathcal{O}}$ to $\mathbb{C}_p$ by projection onto the component
    corresponding to the chosen place above $\frp$.
    We denote the image of an element under this map by $(\,\cdot\,)^*$.
\end{itemize}

\section{Elliptic Soul\'{e} characters}
In this section, we review the definition and basic properties of Kummer characters associated with elliptic units, which we call elliptic Soul\'{e} characters.

\subsection{Fundamental theta function}
First, we recall some basic properties of the fundamental theta function and elliptic
units. We fix an ordered $\mathbb{Z}$-basis $(\omega_1,\omega_2)$ of the
lattice $L$ such that $\tau:=\omega_1/\omega_2$ lies in the upper half-plane.
The Weierstrass $\sigma$-function associated with $L$ is defined by
\[
\sigma(z;L)
=
z\prod_{\lambda\in L\setminus\{0\}}
\left(1-\frac{z}{\lambda}\right)
\exp\left(
\frac{z}{\lambda}
+\frac{1}{2}\left(\frac{z}{\lambda}\right)^2
\right).
\]
The discriminant of $L$ is given by
\[
\Delta(L)
=
\left(\frac{2\boldsymbol{\pi}\boldsymbol{i}}{\omega_2}\right)^{12}
q_\tau\prod_{\nu=1}^{\infty}(1-q_\tau^\nu)^{24},
\qquad
q_\tau=e^{2\boldsymbol{\pi}\boldsymbol{i}\tau}.
\]
We further set
\[
A(L)
:=
\boldsymbol{\pi}^{-1}\mathrm{Area}(\mathbb{C}/L),
\qquad
s_2(L)
:=
\lim_{s\to 0}
\sum_{\lambda\in L\setminus\{0\}}
\lambda^{-2}|\lambda|^{-2s},
\]
where the latter series defines a holomorphic function for
$\operatorname{Re}(s)>0$ and admits a holomorphic continuation to a
neighborhood of $s=0$, with the limit taken after this continuation.
We then define the $\mathbb{R}$-linear function
\[
\eta(z;L):=A(L)^{-1}\bar{z}+s_2(L)z.
\]

\begin{definition}[{\cite[II.2.1(7) and II.2.3]{dS87}}]
    With the above notation, we define the following functions.
    \begin{enumerate}
        \item The fundamental theta function associated with $L$ is defined by
        \[
        \theta(z;L)
        :=
        \Delta(L)\exp\bigl(-6z\eta(z;L)\bigr)\sigma(z;L)^{12}.
        \]

        \item Let $\mathfrak{a}=(\alpha)$ be a nontrivial ideal of
        $O_K$ coprime to $p$. We define
        \begin{equation}\label{eq: product formula for theta_a}
            \theta_{\mathfrak{a}}(z;L)
            :=
            \frac{\theta(z;L)^{\mathrm{N}\mathfrak{a}}}
             {\theta(z;\mathfrak{a}^{-1}L)}
            =
            \alpha^{-12}\Delta(L)^{\mathrm{N}\mathfrak{a}-1}
            \prod_{P\in E[\mathfrak{a}]\setminus\{0\}}
            \bigl(\wp(z;L)-\wp(P;L)\bigr)^{-6}.
        \end{equation}
    \end{enumerate}
\end{definition}

\begin{remark}
    We make some remarks on the functions defined above.
    \begin{itemize}
        \item The fundamental theta function can be written as $\theta(z;L)=\Delta(L)\mathfrak{k}(z,L)^{12}$, where $\mathfrak{k}(z,L) := e^{-\frac{z\eta(z;L)}{2}}\sigma(z;L)$
        is the Klein form associated with $L$, in the notation of Kubert and Lang \cite[p.~27]{KL81}. 
        Moreover, $\theta(z;L)$ coincides with the twelfth power of the Siegel function in their book.

        \item Since $\eta(z;L)$ involves $\bar{z}$, the function $\theta(z;L)$ is not
        holomorphic. Moreover, it is not periodic with respect to $L$. In the
        quotient defining $\theta_{\mathfrak{a}}(z;L)$, however, the non-holomorphic
        factors cancel, and \eqref{eq: product formula for theta_a} shows that
        $\theta_{\mathfrak{a}}(z;L)$ is a meromorphic elliptic function on
        $\mathbb{C}/L$.
    \end{itemize}
\end{remark}

\begin{lemma}
    [{\cite[Lemma 2.3]{IG26}, \cite[II.2.4]{dS87}}]
    \label{lem: algebraicity of theta}
    The following properties hold.
    \begin{enumerate}
        \item Let $\omega\in E[p^\infty]$ be a nonzero $p$-primary torsion
        point, and let $\widetilde{\omega}\in\mathbb{C}$ be any lift of
        $\omega$ under the uniformization
        $\mathbb{C}/L\cong E(\mathbb{C})$. 
        Then $\theta(\widetilde{\omega};L)\in K(p^\infty)^\times$.
        Moreover, the image of $\theta(\widetilde{\omega};L)$ in
        $\widehat{K(p^\infty)}^\times := \varprojlim_n
        K(p^\infty)^\times/K(p^\infty)^{\times p^n}$
        is independent of the choice of $\widetilde{\omega}$. 

        \item Let $\mathfrak{a}$ be a nontrivial ideal of $O_K$ prime to $p$.
        Let $m,n\geq 1$, and let
        $\omega\in E[\frp^n\bar{\frp}^m]$ be primitive 
        (i.e. $\omega\notin E[\frp^{n-1}\bar{\frp}^m]$ and
        $\omega\notin E[\frp^n\bar{\frp}^{m-1}]$).
        Then $\theta_{\mathfrak{a}}(\omega;L)$ is a unit in the ring of
        integers of $K(\frp^n\bar{\frp}^m)$.
    \end{enumerate}
\end{lemma}

\subsection{Elliptic units and elliptic Soul\'e characters}
For each $n\geq 1$ and each pair of integers $0\leq a,b<p^n$ with $p\nmid\gcd(a,b)$, 
choose a lift $\widetilde{\omega}_{a,b} \in \mathbb{C}$ of 
$a\omega_{\frp,n}+b\omega_{\bar{\frp},n}\in E[p^n]$. 
By Lemma \ref{lem: algebraicity of theta}~(i), the class of
$\theta(\widetilde{\omega}_{a,b};L)$ in $K(p^\infty)^\times/K(p^\infty)^{\times p^n}$
is independent of the choice of $\widetilde{\omega}_{a,b}$. We denote this class simply by
$\theta(a\omega_{\frp,n}+b\omega_{\bar{\frp},n};L)$.
For positive integers $i,j$ and each $n\geq 1$, we define
\begin{equation}\label{eq:norm_system}
    \epsilon_{(i,j),n}
    :=
    \prod_{\substack{0\leq a,b<p^n\\p\nmid\gcd(a,b)}}
    \theta(a\omega_{\frp,n}+b\omega_{\bar{\frp},n};L)^{
    a^{i-1}b^{j-1}}
    \in
    K(p^\infty)^\times/K(p^\infty)^{\times p^n}.
\end{equation}

\begin{remark}
    The element $\epsilon_{(i,j),n}$ is independent of the choice of the
    period $\Omega$ satisfying $L=\Omega O_K$. Indeed, replacing $\Omega$
    by $u\Omega$ with $u\in O_K^\times$ replaces
    $\omega_{\frp,n}$ and $\omega_{\bar{\frp},n}$ by their respective
    $u$-multiples. Since $uL=L$, the scaling relation
    \[
    \theta(cz;cL)=\theta(z;L)
    \qquad
    (c\in\mathbb{C}^\times)
    \]
    gives $\theta(uz;L)=\theta(z;L)$. Hence every factor in \eqref{eq:norm_system} remains unchanged.
\end{remark}

By the theory of complex multiplication, $K(p^n)$ is generated over $K$ 
by the values of a Weber function at the $p^n$-torsion points 
(see \cite[Chap.~II, Theorem~5.6]{Sil94}).
Hence $K(p^\infty)$ is the fixed field of $\Delta:=\gal(K(E[p^\infty])/K(p^\infty))$, where $\Delta$ is the kernel of
the composite map
\[
G \xrightarrow{(\kappa_1,\kappa_2)}
O_\frp^\times\times O_{\bar{\frp}}^\times
\longrightarrow
(O_\frp^\times\times O_{\bar{\frp}}^\times)/O_K^\times.
\]
We now assume that $i\equiv j\pmod{w_K}$.
Since the two components of the image of each element under the embedding
$O_K^\times \hookrightarrow O_\frp^\times\times O_{\bar{\frp}}^\times$
are inverses of each other, $\kappa_1^i\kappa_2^j$ factors through
the restriction map $\res\colon G\twoheadrightarrow\mathcal{G}$,
defining a character
\[
\kappa_1^i\kappa_2^j\colon\mathcal{G}\to\mathbb{Z}_p^\times.
\]

\begin{lemma}[{\cite[Lemma 1.2]{IG26}}]
    \label{lem:norm-compatible_system}
    For every $n\geq 1$, the following assertions hold.
    \begin{enumerate}
        \item For every $\sigma\in\mathcal{G}$, one has
        \[
        \sigma\bigl(\epsilon_{(i,j),n}\bigr)
        =
        \epsilon_{(i,j),n}^{
        \kappa_1^{1-i}\kappa_2^{1-j}(\sigma)}.
        \]
        In particular, $\epsilon_{(i,j),n}$ is fixed by $\gal(K(p^\infty)/K(p^n))$.

        \item Under the natural projection
        \[
        K(p^\infty)^\times/K(p^\infty)^{\times p^{n+1}}
        \to
        K(p^\infty)^\times/K(p^\infty)^{\times p^n},
        \]
        the class $\epsilon_{(i,j),n+1}$ maps to
        $\epsilon_{(i,j),n}$.
    \end{enumerate}
\end{lemma}

For each $n \ge 1$, we use the same notation $\kappa_1$ and $\kappa_2$ 
for the mod-$p^n$ characters $\kappa_1,\kappa_2\colon  \gal(K(E[p^n])/K) 
\to (\mathbb{Z}/p^n\mathbb{Z})^\times$.
In particular, when $i\equiv j\pmod{w_K}$, the character
$\kappa_1^i\kappa_2^j$ factors through $\gal(K(p^n)/K)$, and hence defines
\[
\kappa_1^i\kappa_2^j\colon \gal(K(p^n)/K)\to (\mathbb{Z}/p^n\mathbb{Z})^\times.
\]
By the above lemma, the family $\bigl(\epsilon_{(i,j),n}\bigr)_{n\geq 1}$
determines an element of $\widehat{K(p^\infty)}^\times$.
Using this Galois equivariance, we define the associated Kummer characters
as follows.

\begin{definition}\label{def:elliptic_soule}
    Let $(i,j)$ be a pair of positive integers such that $i \equiv j \pmod{w_K}$. 
    \begin{itemize}
        \item We define the $(i, j)$-th elliptic Soul\'e character 
        \[
        \chi_{(i,j)}^\E: G_{K(p^\infty)}^\ab \to \mathbb{Z}_p(i,j) := \mathbb{Z}_p(\kappa_1^i\kappa_2^j)
        \]
        to be the unique character satisfying the following relation for every $n \ge 1$ and $\sigma \in G_{K(p^\infty)}^\ab$: 
        \[
        \zeta_{p^n}^{\chi_{(i,j)}^\E(\sigma)} = \left(\epsilon_{(i,j),n}^{1/p^n}\right)^{\sigma-1}.
        \]
        \item For a nontrivial ideal $\mathfrak{a}$ of $O_K$ prime to $p$,  
        we define
        \[
        \chi_{(i,j),\mathfrak{a}}^{\E}: G_{K(p^\infty)}^\ab \to \mathbb{Z}_p(i,j)
        \]
        to be the unique character satisfying the following relation for every $n \ge 1$ and $\sigma \in G_{K(p^\infty)}^\ab$:
        \[
        \zeta_{p^n}^{\chi_{(i,j),\mathfrak{a}}^{\E}(\sigma)} = \left(\prod_{\tau\in\gal(K(p^n)/K)} \tau\left(\theta_\mathfrak{a}\left(\frac{\Omega}{p^n}; L\right)\right)^{\frac{\kappa_1^{i-1}\kappa_2^{j-1}(\tau)}{p^n}}\right)^{\sigma-1}. 
        \]
    \end{itemize}
\end{definition}
\begin{remark}
    We make some remarks on elliptic Soul\'e characters.
    \begin{itemize}
        \item For integers $i,j \ge 2$ satisfying $i \equiv j \pmod{w_K}$, 
        sufficient conditions for the nontriviality of $\chi_{(i,j)}^\E$
        are given in \cite{Is25}.
        
        \item The elliptic Soul\'{e} characters were first introduced by Nakamura \cite{Na95}. 
        His characters $\kappa_{i,j}\colon G_{K(E[p^\infty])}\to\mathbb{Z}_p$
        are defined by the same Kummer construction as in
        Definition~\ref{def:elliptic_soule}, using the elements given in
        \cite[(3.11.5)]{Na95}.
        Nakamura proved \cite[Theorem~A]{Na95} that these characters describe
        the Galois action on the maximal metabelian quotient of the pro-$p$
        geometric fundamental group of $E\setminus\{O\}$.
        His construction applies to arbitrary elliptic curves over number fields, 
        without the CM assumption.

        \item The difference from Nakamura's formulation lies only in the choice of
        a basis of the $p$-adic Tate module $T_pE$ and a generator of
        $\mathbb{Z}_p(1)$.
        Nakamura uses the basis $(\boldsymbol{\omega}_1,\boldsymbol{\omega}_2)$ of $T_pE$, where
        $\boldsymbol{\omega}_r:=\bigl(\xi(\omega_r/p^n;L)\bigr)_n$ for $r=1,2$,
        and the generator
        $\boldsymbol{\xi}:=\bigl(\exp(2\boldsymbol{\pi}\boldsymbol{i}/p^n)\bigr)_n$ of $\mathbb{Z}_p(1)$,
        whereas we use
        $\boldsymbol{\omega}_{\frp}:=(\omega_{\frp,n})_n$,
        $\boldsymbol{\omega}_{\bar{\frp}}:=(\omega_{\bar{\frp},n})_n$, and
        $\boldsymbol{\zeta}:=(\zeta_{p^n})_n$.
        By restricting to $G_{K(E[p^\infty])}$, we may define
        $\chi_{(i,j)}^\E$ for arbitrary positive integers $i$ and $j$
        by the same Kummer relation.
        A direct calculation under the corresponding change of basis shows
        that
        \[
        \sum_{i+j=m}
        \kappa_{i,j}\,
        \boldsymbol{\omega}_1^{\otimes i}
        \otimes
        \boldsymbol{\omega}_2^{\otimes j}
        \otimes
        \boldsymbol{\xi}
        \quad\text{and}\quad
        \sum_{i+j=m}
        \chi_{(i+1,j+1)}^\E\,
        \boldsymbol{\omega}_{\frp}^{\otimes i}
        \otimes
        \boldsymbol{\omega}_{\bar{\frp}}^{\otimes j}
        \otimes
        \boldsymbol{\zeta}
        \]
        define the same element of 
        $H^1\!\left(
        G_{K(E[p^\infty])},
        \operatorname{Sym}^m T_pE(1)
        \right)^{\gal(K(E[p^\infty])/K)}$.
    \end{itemize}
\end{remark}

\begin{proposition}[{\cite[Proposition 4.2]{Is25}}]\label{prop: compare character}
    Let $\mathfrak{a}$ be a nontrivial ideal of $O_K$ prime to $p$, and let
    $\sigma_{\mathfrak{a}} = \left(\mathfrak{a}, K(p^\infty)/K\right) \in \gal(K(p^\infty)/K)$ 
    be the corresponding Artin symbol.
    Then, for each pair of positive integers $(i,j) \neq (1,1)$ such that 
    $i \equiv j \pmod{w_K}$, the following 
    relation holds: 
    \[
    w_K \left(\frac{1}{1 - \pi^{i-1}\bar{\pi}^{j-1}} + \frac{1}{1 - \pi^{j-1}\bar{\pi}^{i-1}} - 1\right) \chi_{(i,j),\mathfrak{a}}^\E = (-1)^i\left(\kappa_1^{1-i}\kappa_2^{1-j}(\sigma_{\mathfrak{a}}) - \operatorname{N}\mathfrak{a}\right) \chi_{(i,j)}^\E.
    \]
\end{proposition}

\section{Coleman theory in the cyclotomic and CM elliptic settings}

In this section, we review the theory of Coleman power series in both 
the cyclotomic and elliptic settings, and introduce the two-variable 
Coates--Wiles homomorphisms and the two-variable 
$p$-adic $L$-function appearing in our formula.

\subsection{Coleman power series in the cyclotomic setting}
In this subsection, we recall Coleman power series in the cyclotomic setting
and Coleman's explicit reciprocity law, which expresses Hilbert norm residue
symbols in terms of these power series. 

First, we set
\[
\widehat{\mathcal{O}}[[S]]^{\times1}
:=
\left\{
f(S)\in\widehat{\mathcal{O}}[[S]]^\times
\mathrel{}\middle|\mathrel{}
f(0)\equiv1\pmod{p}
\right\}.
\]
For $f(S)\in\widehat{\mathcal{O}}[[S]]^{\times1}$ and
$c\in\mathbb{Z}_p$, define
\[
f(S)^c
:=
\sum_{r=0}^{\infty}
\binom{c}{r}\bigl(f(S)-1\bigr)^r,
\qquad
\binom{c}{r}
:=
\frac{c(c-1)\cdots(c-r+1)}{r!}.
\]
Since $f(S)-1\in(p,S)\widehat{\mathcal{O}}[[S]]$,
this series converges in the $(p,S)$-adic topology and defines an element
of $\widehat{\mathcal{O}}[[S]]^{\times1}$.
For $\sigma\in G$ and $f(S)\in\widehat{\mathcal{O}}[[S]]$, let
$f^\sigma(S)$ denote the power series obtained by applying $\sigma$ to the
coefficients of $f(S)$. We define a left action of $G$ on
$\widehat{\mathcal{O}}[[S]]$ by
\[
(\sigma\cdot f)(S)
:=
f^\sigma\left(
[\chi_\cyc(\sigma)]_{\widehat{\mathbb{G}}_m}(S)
\right).
\]
This action preserves $\widehat{\mathcal{O}}[[S]]^{\times1}$. Hence, for
$f(S)\in\widehat{\mathcal{O}}[[S]]^{\times1}$ and a finite sum
$\sum_\sigma c_\sigma\sigma\in\mathbb{Z}_p[G]$, we may define
\[
f(S)^{\sum_\sigma c_\sigma\sigma}
:=
\prod_\sigma
\bigl((\sigma\cdot f)(S)\bigr)^{c_\sigma}.
\]
For each $m\geq0$, set $\mathcal{O}_m[[S]]^{\times1}
:= \mathcal{O}_m[[S]]^\times \cap \widehat{\mathcal{O}}[[S]]^{\times1}$.

\begin{definition}[{\cite[I.2.2]{dS87}}]
    Let $m\geq0$ be an integer. For every norm-compatible system
    $\boldsymbol{\epsilon}_m=(\epsilon_{m,n})_n\in\mathbf{U}_m^\cyc$, the
    Coleman power series associated with $\boldsymbol{\epsilon}_m$ is the
    unique power series $f_{\boldsymbol{\epsilon}_m}(S) 
    \in\mathcal{O}_m[[S]]^{\times1}$ satisfying
    \[
    f_{\boldsymbol{\epsilon}_m}^{\phi^{-n}}
    (\zeta_{p^n}-1)
    =
    \epsilon_{m,n}
    \]
    for every $n\geq1$.
\end{definition}

Next, we introduce several operators on the power series ring.
The norm operator introduced in \cite{Col79} and the reduced trace operator 
introduced in \cite[p.~386]{Col81} can be defined in the same way over $\widehat{\mathcal{O}}$.

\begin{definition}[{\cite{Col79}, \cite[p.~386]{Col81}}]
\leavevmode
\begin{enumerate}
    \item The norm operator $\cnr\colon \widehat{\mathcal{O}}[[S]]^\times
    \to \widehat{\mathcal{O}}[[S]]^\times$
    is the unique multiplicative operator satisfying
    \[
    ([p]_{\widehat{\mathbb{G}}_m}\circ\cnr)(f)(S)
    =
    \prod_{\xi\in\mu_p} f(\xi(1+S)-1).
    \]

    \item The reduced trace operator $\ctr\colon \widehat{\mathcal{O}}[[S]]
    \to \widehat{\mathcal{O}}[[S]]$ is the unique additive operator satisfying
    \[
    ([p]_{\widehat{\mathbb{G}}_m}\circ\ctr)(f)(S)
    =
    \frac{1}{p}\sum_{\xi\in\mu_p} f(\xi(1+S)-1).
    \]

    \item We define the integral logarithm
    $\intlog\colon \widehat{\mathcal{O}}[[S]]^\times
    \to \widehat{\mathcal{O}}[[S]]$ by
    \[
    \intlog(f)(S)
    :=
    \frac{1}{p}
    \log\left(
        \frac{f(S)^p}{f^\phi((1+S)^p-1)}
    \right).
    \]

    \item We further define the differential operator
    $D\colon \widehat{\mathcal{O}}[[S]] \to \widehat{\mathcal{O}}[[S]]$
    by
    \[
    D:=(1+S)\frac{d}{dS}.
    \]
\end{enumerate}
\end{definition}

A simple computation shows that the integral logarithm restricts to a homomorphism
\begin{equation}\label{eq:integral_log}
    \intlog\colon (\widehat{\mathcal{O}}[[S]]^\times)^{\cnr=\phi}
    \to \widehat{\mathcal{O}}[[S]]^{\ctr=0}.
\end{equation}
For $f\in \widehat{\mathcal{O}}[[S]]$, let $\mu_f$ denote the measure
corresponding to $f$ under the Amice transform, so that
\begin{equation}\label{eq:amice_transform}
    f(S)=\int_{\mathbb{Z}_p}(1+S)^x\,d\mu_f(x).
\end{equation}

\begin{lemma}\label{lem: fundamental properties of Coleman power ser.}
    The following properties hold: 
    \begin{enumerate}
        \item For any $\sigma\in G$, $f_{\sigma(\boldsymbol{\epsilon}_m)}(S) = \sigma \cdot f_{\boldsymbol{\epsilon}_m}(S)$. 
        \item $\cnr(f_{\boldsymbol{\epsilon}_m})(S) = (f_{\boldsymbol{\epsilon}_m}^{\phi})(S)$. 
        \item The differential operator induces an isomorphism 
        $D\colon \widehat{\mathcal{O}}[[S]]^{\ctr=0} \xrightarrow{\sim} \widehat{\mathcal{O}}[[S]]^{\ctr=0}$. 
        Moreover, for each $f\in \widehat{\mathcal{O}}[[S]]^{\ctr=0}$ and each $k\in\mathbb{Z}$, one has
        \[
        D^k(f)(0) = \int_{\mathbb{Z}_p^\times} x^k\,d\mu_f(x).
        \]
    \end{enumerate}
\end{lemma}

\begin{proof}
    Assertions (i) and (ii) are the semilocal versions of
    \cite[II.2.3, Corollary (iv) and (ii)]{dS87}, respectively,
    while (iii) is the semilocal version of \cite[Lemma 3.4]{NSW17a}.
    The proofs are the same as in the local setting.
\end{proof}

For integers $m,n\ge 0$, let 
\[
\mathcal{F}_m(\mu_{p^n}) := F_m(\mu_{p^n})\otimes_K K_{\frp} \cong \prod_{v\mid\frp}F_{m,v}(\mu_{p^n}). 
\]
For each place $v$ of $F_m$ above $\frp$, let
$\rec_{m,n,v}\colon F_{m,v}(\mu_{p^n})^{\times} \to G_{F_{m,v}(\mu_{p^n})}^{\ab}$
denote the local reciprocity map. 
Taking the product of these maps, we obtain 
\[
\rec_{m,n}\colon \mathcal{F}_m(\mu_{p^n})^{\times} \to \prod_{v\mid\frp}G_{F_{m,v}(\mu_{p^n})}^{\ab},
\quad (a_v)_v \mapsto \bigl(\rec_{m,n,v}(a_v)\bigr)_v.
\]
By the norm compatibility of local reciprocity, these maps induce a
continuous injection
\[
\rec_\infty\colon \mathbf{U}_\infty^\cyc \hookrightarrow \prod_{v\mid\frp}G_{K(E[p^\infty])_v}^{\ab}, 
\]
where the product is taken over the places $v$ of $K(E[p^\infty])$ above $\frp$.
For each such place $v$, choose an extension of $v$ to $\overline{K}$ and identify 
$G_{K(E[p^\infty])_v}$ with the corresponding decomposition subgroup of $G_{K(E[p^\infty])}$.  
Let $\lambda_v\colon G_{K(E[p^\infty])_v}^{\ab} \to G_{K(E[p^\infty])}^{\ab}$
denote the homomorphism induced on abelianizations by the inclusion of this decomposition subgroup.
Since different choices of an extension of $v$ give conjugate embeddings,
the homomorphism $\lambda_v$ is independent of this choice.
Since there are only finitely many places of $K(E[p^\infty])$ above $\frp$, we define
\begin{equation}\label{eq:embedding_to_global}
    \lambda_\frp\colon \prod_{v\mid\frp} G_{K(E[p^\infty])_v}^{\ab} \to G_{K(E[p^\infty])}^{\ab},
    \quad (\sigma_v)_v \mapsto \prod_{v\mid\frp}\lambda_v(\sigma_v).
\end{equation}

\begin{proposition}\label{prop:injectivity-of-reciprocity-map}
    The composition
    \[
    \lambda_\frp \circ \rec_\infty\colon \mathbf{U}_\infty^\cyc \to G_{K(E[p^\infty])}^\ab 
    \]
    is injective. 
\end{proposition}

\begin{proof}
    For each $m, n \ge 0$, let $\lambda_{m,n}\colon\prod_{v\mid\frp} G_{F_{m,v}(\mu_{p^n})}^\ab \to G_{F_m(\mu_{p^n})}^\ab$
    be the homomorphism induced by the inclusions of the corresponding
    decomposition subgroups.
    Since $U_{m,n}^\cyc$ is a pro-$p$ group, 
    \cite[Chap.~III, Corollary~4.16.7]{Gra03} implies that 
    $\lambda_{m,n} \circ\rec_{m,n}\colon U_{m,n}^\cyc \hookrightarrow G_{F_m(\mu_{p^n})}^\ab$
    is injective. 
    Passing to the inverse limit, we obtain the desired injectivity.
\end{proof}

Via the isomorphism $\mathbf{U}_\infty^\E\cong\mathbf{U}_\infty^\cyc$
from Section~\ref{sec:notation}, we view $\rec_\infty$ as a map on
$\mathbf{U}_\infty^\E$, and use the same notation:
\begin{equation}\label{eq:reciprocity-map}
    \rec_\infty\colon \mathbf{U}_\infty^\E \hookrightarrow \prod_{v\mid\frp} G_{K(E[p^\infty])_v}^\ab.
\end{equation}

\begin{definition}\label{def:Hilbert_symbol}
    The Hilbert norm residue symbol
    \[
    (\cdot,\cdot)_{m,n}\colon \mathcal{F}_m(\mu_{p^n})^\times \times \mathcal{F}_m(\mu_{p^n})^\times \to \mu_{p^n}
    \]
    is defined by 
    \[
    (a, b)_{m,n} := \prod_{v\mid\frp} (a_v, b_v)_{m,n,v}, \qquad 
    (a_v, b_v)_{m,n,v} := (a_v^{1/p^n})^{\rec_{m,n,v}(b_v)-1}
    \]
    for $a = (a_v)_v$ and $b = (b_v)_v$ in $\mathcal{F}_m(\mu_{p^n})^\times$. 
\end{definition}

We now state Coleman's explicit reciprocity law in a semilocal form.
For $f(S)\in\widehat{\mathcal{O}}[[S]]$, we write $f_v(S)$ for the power series 
obtained by projecting each coefficient onto the component corresponding to the 
place $v$.
Following Coleman \cite{Col81}, we introduce a continuous linear functional 
$\int_n = (\int_{n,v})_v\colon \mathcal{O}_m[[S]] \to \mathcal{O}_m$ by
\[
\int_n f := \frac{1}{p^n}\sum_{\zeta\in\mu_{p^n}} f(\zeta-1) = \left(\frac{1}{p^n}\sum_{\zeta\in\mu_{p^n}} f_v(\zeta - 1)\right)_v. 
\]
\begin{theorem}[\cite{Col81}, \cite{Col89}]\label{thm: Coleman}
    Let $f(S) \in (1 + S\mathcal{O}_m[[S]])^{\mathcal{N}=\phi}$, 
    and let $g_{\boldsymbol{\epsilon}_m}(S)$ be the Coleman 
    power series associated with $\boldsymbol{\epsilon}_m = (\epsilon_{m,n})_n 
    \in \mathbf{U}_m^\cyc$. 
    Then, for every $n\ge1$,
    \[
    (f(\zeta_{p^n}-1), \epsilon_{m,n})_{m,n} = \zeta_{p^n}^{\tr_{\mathcal{F}_m/K_{\frp}}(\int_n \intlog(f)\cdot D\intlog(\phi^{-n}g_{\boldsymbol{\epsilon}_m}))}, 
    \]
    where $\tr_{\mathcal{F}_m/K_\frp}$ denotes the trace map
    \[ 
    \tr_{\mathcal{F}_m/K_\frp}(x) = \sum_{v\mid\frp} \tr_{F_{m,v}/K_\frp}(x_v) \qquad (x = (x_v)_v \in \mathcal{F}_m). 
    \]
\end{theorem}
\begin{proof}
    For each $n\ge 0$, write 
    $\epsilon_{m,n}=(\epsilon_{m,n,v})_v\in\prod_{v\mid\frp}U_{F_{m,v}(\mu_{p^n})}$, 
    and for each $v\mid\frp$, set $\boldsymbol{\epsilon}_v:=(\epsilon_{m,n,v})_n
    \in\varprojlim_n U_{F_{m,v}(\mu_{p^n})}$.
    Applying Coleman's explicit reciprocity law in the form given in
    \cite[Theorem~3.7]{NSW17a}, we obtain
    \[
    (f_v(\zeta_{p^n}-1), \epsilon_{m,n,v})_{m,n,v} = 
    \zeta_{p^n}^{\tr_{F_{m,v}/K_\frp}\left(\int_{n,v} \intlog(f_v)\cdot D\intlog(\phi^{-n}g_{\boldsymbol{\epsilon}_v})\right)}. 
    \]
    Taking the product over all places above $\frp$ gives the desired equality.
\end{proof}

\subsection{Coleman power series in the CM elliptic setting}

In this subsection, we review the Coleman power series attached to $\widehat{E}$ and their basic properties. 
Using these power series, we define the two-variable Coates--Wiles homomorphisms and recall de Shalit's two-variable $p$-adic $L$-function.

\begin{proposition}[{\cite[II.4.5]{dS87}}]\label{prop: CPS a.w. ell}
    For each norm-compatible system $\boldsymbol{\beta}_m = (\beta_{m,n})_n \in \mathbf{U}_m^\E$, there exists a unique power series $g_{\boldsymbol{\beta}_m}(T) \in \mathcal{O}_m[[T]]^\times$ satisfying the interpolation property:
    \[
    g_{\boldsymbol{\beta}_m}^{\phi^{-n}}(\omega_n) = \beta_{m,n}
    \]
    for all $n \ge 1$. 
    The following properties hold:
    \begin{enumerate}
        \item $\beta_{m,0} = g_{\boldsymbol{\beta}_m}^{1-\phi^{-1}}(0)$. 
        \item $g_{\boldsymbol{\beta}_m}^\phi \circ [\pi]_{\widehat{E}}(T) = \prod_{\omega\in \widehat{E}[\frp]} g_{\boldsymbol{\beta}_m}(\omega +_{\widehat{E}} T)$. 
        \item For any $\sigma \in G$, $g_{\sigma(\boldsymbol{\beta}_m)}(T) = g_{\boldsymbol{\beta}_m}^\sigma\circ[\kappa_1(\sigma)]_{\widehat{E}}(T)$. 
    \end{enumerate}
\end{proposition}

By \cite[II.4.3]{dS87}, there exists an isomorphism of formal groups
\begin{equation}\label{eq:eta}
    \eta\colon \widehat{\mathbb{G}}_m \xrightarrow{\sim} \widehat{E},
    \qquad T=\eta(S)=\Omega_pS+\cdots \in \widehat{\mathcal{O}}[[S]],
\end{equation}
defined over $\widehat{\mathcal{O}}$ and satisfying
\begin{equation}\label{eq:eta_compatibility}
    [\kappa_1(\sigma)]_{\widehat{E}} \circ \eta = \eta^\sigma \circ [\chi_{\cyc}(\sigma)]_{\widehat{\mathbb{G}}_m} \qquad (\sigma\in G).
\end{equation}
Such an isomorphism $\eta$ is unique up to right composition with
$[a]_{\widehat{\mathbb{G}}_m}$ for some $a\in O_\frp^\times$.
Henceforth, we fix $\eta$, and hence $\Omega_p$, by imposing the normalization (see {\cite[II.4.4]{dS87}})
\begin{equation}\label{eq:normalization}
    \omega_n=\eta^{\phi^{-n}}(\zeta_{p^n}-1) \qquad (n\ge1).
\end{equation}
In what follows, we also write $\Omega_p$ for $\Omega_p^*\in \mathbb{C}_p$.

\begin{lemma}\label{lem: norm relation for elliptic Coleman series}
Let $m\ge 1$ and $\boldsymbol{\beta}_m\in \mathbf{U}_m^\E$. 
Then the following hold.
\begin{enumerate}
    \item $[\pi]_{\widehat{E}}\circ \eta^{\phi^{-1}} = \eta\circ [p]_{\widehat{\mathbb{G}}_m}$.
    \item $\cnr(g_{\boldsymbol{\beta}_m}\circ\eta) = (g_{\boldsymbol{\beta}_m}\circ\eta)^\phi$.
\end{enumerate}
\end{lemma}
\begin{proof}
    The identity (i) follows from the normalization \eqref{eq:normalization} by
    comparing the values of the two homomorphisms at the torsion points
    $\zeta_{p^n}-1$ for $n\ge 1$.
    For (ii), using the definition of $\cnr$ and (i), we have
    \[
    ([p]_{\widehat{\mathbb{G}}_m}\circ \cnr)((g_{\boldsymbol{\beta}_m}\circ\eta)^{\phi^{-1}})(S)
    =
    \prod_{\omega\in\widehat{E}[\frp]}
    g_{\boldsymbol{\beta}_m}^{\phi^{-1}}
    \bigl(\omega+_{\widehat{E}}\eta^{\phi^{-1}}(S)\bigr).
    \]
    Applying $\phi^{-1}$ to Proposition~\ref{prop: CPS a.w. ell} (ii), and using
    (i) again, we obtain
    \[
    ([p]_{\widehat{\mathbb{G}}_m}\circ \cnr)((g_{\boldsymbol{\beta}_m}\circ\eta)^{\phi^{-1}})
    =
    [p]_{\widehat{\mathbb{G}}_m}(g_{\boldsymbol{\beta}_m}\circ\eta).
    \]
    Since $[p]_{\widehat{\mathbb{G}}_m}$ is injective on
    $\widehat{\mathcal{O}}[[S]]$, the desired equality follows.
\end{proof}

For each integer $m\ge 1$, set
\[
G_m := \gal(F_m(E[\frp^\infty])/K), \qquad H_m := \gal(F_m(E[\frp^\infty])/F_m).
\]
\begin{proposition}[{\cite[II.4.6]{dS87}}]\label{prop: measure existence}
    For each \(m\ge 1\), there exists a unique $G_m$-equivariant homomorphism
    $i_m:\mathbf{U}_m^\E \to \Lambda(G_m,\widehat{\mathcal O}), 
    \;\boldsymbol{\beta}_m \mapsto \mu_{\boldsymbol{\beta}_m}$
    such that, for every \(\boldsymbol{\beta}_m\in \mathbf{U}_m^\E\), 
    \[
    \intlog(g_{\boldsymbol{\beta}_m}\circ\eta)(S)
    =
    \int_{H_m}
    (1+S)^{\kappa_1(\sigma)}
    \,d\mu_{\boldsymbol{\beta}_m}(\sigma),
    \]
    where $\kappa_1$ restricts to an isomorphism
    $\kappa_1\colon H_m\xrightarrow{\sim}\mathbb{Z}_p^\times$.
    The $G_m$-equivariance of $i_m$ is given by $\mu_{\tau(\boldsymbol{\beta}_m)}(\;\cdot\;) = \mu_{\boldsymbol{\beta}_m}(\tau^{-1}\;\cdot\;)$ for any $\tau \in G_m$. 
\end{proposition}

For integers $m'\ge m\ge 1$, let $\res_{m',m}\colon G_{m'} \twoheadrightarrow G_m$ be the restriction map, 
and denote by $(\res_{m',m})_*\colon \Lambda(G_{m'}, \widehat{\mathcal{O}}) \to \Lambda(G_m,\widehat{\mathcal{O}})$ 
the induced push-forward of measures. 
Also let $\nr_{m',m}\colon \mathbf{U}_{m'}^\E \to \mathbf{U}_m^\E$ be the norm map.
By the norm compatibility of the Coleman power series and the uniqueness 
argument in Proposition~\ref{prop: measure existence}, we have
$(\res_{m',m})_*\circ i_{m'} = i_m\circ\nr_{m',m}$.
Passing to the inverse limit, we obtain a $G$-equivariant homomorphism
\begin{equation}\label{eq:two-variable_measure}
    i_\infty\colon \mathbf{U}_\infty^\E
    \to \Lambda(G,\widehat{\mathcal O}) \cong
    \varprojlim_m\Lambda(G_m,\widehat{\mathcal O}),
    \; (\boldsymbol{\beta}_m)_m \mapsto \mu_{\boldsymbol{\beta}} := (\mu_{\boldsymbol{\beta}_m})_m. 
\end{equation}
With the notation of Section~\ref{sec:notation},
$\mu_{\boldsymbol{\beta}}^*\in\Lambda(G,\mathbb C_p)$ denotes the corresponding $\mathbb{C}_p$-valued measure.

By Lemma~\ref{lem: norm relation for elliptic Coleman series}~(ii)
and \eqref{eq:integral_log}, we have
$\ctr\bigl(\intlog(g_{\boldsymbol{\beta}_m}\circ\eta)\bigr)=0$.
Therefore, Lemma~\ref{lem: fundamental properties of Coleman power ser.} (iii) implies that, for every integer $k\in\mathbb{Z}$, we may define $D^k\intlog(g_{\boldsymbol{\beta}_m}\circ\eta)(0)$.
Moreover, for each $k \ge 1$, we define Kummer's $k$-th logarithmic derivative of $\boldsymbol{\beta}_m \in \mathbf{U}_m^\E$ by 
\[
\delta_k(\boldsymbol{\beta}_m) = D^k(\log g_{\boldsymbol{\beta}_m} \circ \eta)(0). 
\]
We note that $\delta_k(\boldsymbol{\beta}_m) \in \widehat{\mathcal{O}}$ for every $k \ge 1$, since  $(g_{\boldsymbol{\beta}_m}\circ\eta)(S) \in \widehat{\mathcal{O}}[[S]]^\times$ implies that $(g_{\boldsymbol{\beta}_m}\circ\eta)'(S)/(g_{\boldsymbol{\beta}_m}\circ\eta)(S) \in \widehat{\mathcal{O}}[[S]]$. 

\begin{lemma}\label{lem: approximation}
Let \(\boldsymbol{\beta}=(\boldsymbol{\beta}_m)_m\in \mathbf{U}_\infty^\E\)
be a norm-compatible system.
For each pair of integers $(i,j)$, we have
\[
\int_{G} \kappa_1^i\kappa_2^j(\sigma)\,d\mu_{\boldsymbol{\beta}}^*(\sigma) \equiv \sum_{\tau\in \gal(F_m/K)} \kappa_2^{-j}(\tau)\, D^i\intlog(g_{\widetilde{\tau}(\boldsymbol{\beta}_m)}\circ\eta)(0)^* \pmod{p^m}, 
\]
where $\widetilde{\tau}\in G$ denotes a lift of
$\tau$ acting trivially on $E[\frp^\infty]$.
Moreover, when $i \ge 1$, 
\[
\int_G \kappa_1^i\kappa_2^j(\sigma)\,d\mu_{\boldsymbol{\beta}}^*(\sigma) \equiv (1 - \pi^{j-1}\bar{\pi}^{i-1}) \sum_{\tau \in \gal(F_m/K)} \kappa_2^{-j}(\tau) \delta_i(\widetilde{\tau}(\boldsymbol{\beta}_m))^* \pmod{p^m}. 
\]
The right-hand sides of both congruences are independent of the choice of
the lifts $\widetilde{\tau}$.
\end{lemma}

\begin{proof}
    Using the lifts $\widetilde{\tau}\in G$ chosen above, 
    we have the coset decomposition
    \[
        G_m = \bigsqcup_{\tau\in\gal(F_m/K)} \widetilde{\tau}^{-1}H_m.
    \]
    Using this decomposition and the equivariance property 
    in Proposition~\ref{prop: measure existence}, we obtain
    \[
    \int_{G} \kappa_1^i\kappa_2^j(\sigma)\, d\mu_{\boldsymbol{\beta}}^*(\sigma)
    \equiv \sum_{\tau\in \gal(F_m/K)} \kappa_2^{-j}(\tau) \int_{H_m} 
    \kappa_1^i(\sigma)\, d\mu_{\widetilde{\tau}(\boldsymbol{\beta}_m)}^*(\sigma) 
    \pmod{p^m}.
    \]
    Therefore, the first congruence follows by applying $D^i$ to the left-hand side of the identity in Proposition~\ref{prop: measure existence} and then evaluating at $S=0$.
    For the second congruence, by the definition of the integral logarithm, we have
    \[
    D^i\intlog(g_{\widetilde{\tau}(\boldsymbol{\beta}_m)} \circ \eta)(0) = \delta_i(\widetilde{\tau}(\boldsymbol{\beta}_m)) - p^{i-1}\phi(\delta_i(\widetilde{\tau}(\boldsymbol{\beta}_m))).
    \]
    For the second term on the right-hand side, the identities
    $\chi_\cyc(\phi)=\kappa_2(\phi)=\pi$, together with
    Proposition~\ref{prop: CPS a.w. ell} (iii) and
    \eqref{eq:eta_compatibility}, yield
    $\phi\left(\delta_i(\widetilde{\tau}(\boldsymbol{\beta}_m))\right)
    =\pi^{-i}\delta_i(\phi\widetilde{\tau}(\boldsymbol{\beta}_m))$.
    Therefore, we have
    \[
    \pi^{-i}\sum_{\tau\in \gal(F_m/K)} \kappa_2^{-j}(\tau)\delta_i(\phi\widetilde{\tau}(\boldsymbol{\beta}_m)) = \pi^{j-i} \sum_{\tau\in \gal(F_m/K)} \kappa_2^{-j}(\tau) \delta_i(\widetilde{\tau}(\boldsymbol{\beta}_m)). 
    \]
    Hence, the second congruence holds.
    The independence of the choices of the lifts $\widetilde{\tau}$ follows immediately from the argument above.
\end{proof}

\begin{lemma}
    Let $m\ge 1$ and $\boldsymbol{\beta}_m\in \mathbf{U}_m^\E$. 
    For each pair of positive integers $(i,j)$, we write
    \[
    \delta_{i,j}^{(m)}(\boldsymbol{\beta}_m) := \sum_{\tau\in\gal(F_m/K)} \kappa_2^{-j}(\tau) \delta_i(\widetilde{\tau}(\boldsymbol{\beta}_m)) \in \widehat{\mathcal{O}}/p^m\widehat{\mathcal{O}},
    \]
    where $\widetilde{\tau}\in G$ denotes a lift of
    $\tau$ acting trivially on $E[\frp^\infty]$.
    Then the following properties hold:
    \begin{enumerate}
        \item For every $\sigma \in G$, $\delta_{i,j}^{(m)}(\sigma(\boldsymbol{\beta}_m)) \equiv \kappa_1^i\kappa_2^j(\sigma)\delta_{i,j}^{(m)}(\boldsymbol{\beta}_m) \pmod{p^m}$.
        \item $\Omega_p^{j-i}\delta_{i,j}^{(m)}(\boldsymbol{\beta}_m)^* \in \mathbb{Z}_p/p^m\mathbb{Z}_p$. 
    \end{enumerate}
\end{lemma}

\begin{proof}
    (i) Given $\sigma \in G$, take $\widetilde{\sigma} \in G$ such that
    $\kappa_2(\sigma)=\kappa_2(\widetilde{\sigma})$ and
    $\widetilde{\sigma}$ acts trivially on $E[\frp^\infty]$.
    Since $\sigma(\delta_i(\boldsymbol{\beta}_m)) =
    \widetilde{\sigma}(\delta_i(\boldsymbol{\beta}_m))$,
    Proposition~\ref{prop: CPS a.w. ell}~(iii) and \eqref{eq:eta_compatibility} give
    $\delta_i(\sigma(\boldsymbol{\beta}_m)) = \kappa_1^i(\sigma)\delta_i(\widetilde{\sigma}(\boldsymbol{\beta}_m))$. 
    Therefore, we obtain
    \[
    \delta_{i,j}^{(m)}(\sigma(\boldsymbol{\beta}_m)) 
    \equiv \kappa_1^i(\sigma) \sum_{\tau\in\gal(F_m/K)} \kappa_2^{-j}(\tau) \delta_i(\widetilde{\sigma}\widetilde{\tau}(\boldsymbol{\beta}_m)) 
    \equiv \kappa_1^i\kappa_2^j(\sigma)\delta_{i,j}^{(m)}(\boldsymbol{\beta}_m) \pmod{p^m}. 
    \]
    (ii) The same equivariance relation gives $\sigma(\delta_{i,j}^{(m)}(\boldsymbol{\beta}_m)) = \kappa_2^{j-i}(\sigma)\delta_{i,j}^{(m)}(\boldsymbol{\beta}_m)$ for every $\sigma \in G$.
    On the other hand, \eqref{eq:eta_compatibility} implies $\sigma(\Omega_p) = \kappa_2^{-1}(\sigma)\Omega_p$. 
    Hence, $\Omega_p^{j-i}\delta_{i,j}^{(m)}(\boldsymbol{\beta}_m)$ is fixed by $G$. 
\end{proof}

\begin{definition}\label{def:coates-wiles}
    For each pair of positive integers $(i,j)$,  
    we define the $(i,j)$-th Coates--Wiles homomorphism
    \[
    \phi_{(i,j)}^\cw: \mathbf{U}_\infty^\E \to \mathbb{Z}_p(i,j)
    \]
    by the congruences
    \[
    \phi_{(i,j)}^{\cw}(\boldsymbol{\beta}) \equiv \Omega_p^{j-i} \sum_{\tau\in \gal(F_m/K)} \kappa_2^{-j}(\tau) \delta_i(\widetilde{\tau}(\boldsymbol{\beta}_m))^* \pmod{p^m}
    \]
    for every $m\ge 1$, where $\widetilde{\tau}\in G$ denotes a lift of
    $\tau$ acting trivially on $E[\frp^\infty]$.
\end{definition}

\begin{remark}
    By Lemma~\ref{lem: approximation}, the right-hand side of the above congruence is well-defined.
    The homomorphism defined above was originally introduced by Yager in \cite[(14)]{Yag82}.
    Although Yager defines the first component using the differential
    operator on $\widehat{E}$, his formulation coincides with ours after
    pullback along the isomorphism $\eta$ defined in \eqref{eq:eta}, up to
    the normalization factor $\Omega_p^j$ incorporated into our definition.
\end{remark}

We now introduce the two-variable $p$-adic $L$-function following de Shalit's formulation.
Let $\varphi$ be a Gr\"ossencharacter of $K$ of conductor prime to $\frp$ 
and of infinity type $(1,0)$. 
Given a Gr\"ossencharacter of the form $\varepsilon = \chi\varphi^i\bar{\varphi}^j$, 
where $\chi$ is a finite-order character, write $\frp^n$ for the $\frp$-part
of the conductor of $\chi$. 
Regarding $\chi$ as a Galois character via the Artin reciprocity map, 
define a generalized Gauss sum \cite[II.4, (30)]{dS87} by
\[ 
G(\varepsilon) := \frac{\bigl(\varphi^i\bar{\varphi}^j\bigr)(\frp^n)}{p^n} \sum_{\gamma\in C_n} \chi(\gamma)\bigl(\zeta_{p^n}^{\gamma}\bigr)^{-1}, 
\]
where 
\[ 
C_n := \left\{ \gamma\in\gal(K(\frp^n\bar{\frp}^{\infty})/K) \ \middle|\ \left.\gamma\right|_{K(\bar{\frp}^{\infty})} = (\frp^n,K(\bar{\frp}^{\infty})/K) \right\}. 
\] 

\begin{theorem}[{\cite[II.4.14]{dS87}}]\label{thm:interpolation_properties} 
    There exists a unique measure $ \mu_{\mathrm{dS}} \in \Lambda(\mathcal{G},\widehat{\mathcal{O}})$ satisfying the interpolation formula 
    \[ 
    \Omega_p^{j-i} \int_{\mathcal{G}} \varepsilon(\sigma)\,d\mu_{\mathrm{dS}}^*(\sigma) 
    = \Omega^{j-i} (i-1)!\left(\frac{\sqrt{d_K}}{2\boldsymbol{\pi}}\right)^j G(\varepsilon) \left(1-\frac{\varepsilon(\frp)}{p}\right) \left(1 - \varepsilon^{-1}(\bar{\frp})\right) L(\varepsilon^{-1},0) 
    \] 
    for any Gr\"ossencharacter $\varepsilon$ with conductor dividing
    $p^\infty$ and infinity type $(i,j)$ satisfying $0\le -j<i$. 
    Here, $L(\varepsilon^{-1},s)$ denotes the 
    complex Hecke $L$-function associated with $\varepsilon^{-1}$. 
\end{theorem}

\begin{remark}
As noted by Darmon--Lauder--Rotger \cite[p.~41, after (46)]{DLR15},
a minor correction to the power of $2\boldsymbol{\pi}$ appearing in de Shalit's
interpolation formula \cite[II.4.14]{dS87} is required, and the
formula stated above incorporates this correction.
\end{remark}

Under the standard decomposition 
$\mathbb{Z}_p^\times \cong \mu_{p-1}\times(1+p\mathbb{Z}_p)$, 
each $a\in\mathbb{Z}_p^\times$ can be written uniquely as 
$a=\omega(a)\langle a\rangle$, where $\omega$ is the Teichm\"uller 
character and $\langle\cdot\rangle$ denotes the projection onto 
$1+p\mathbb{Z}_p$. 
Since the images of  $\Delta\cong\gal(K(E[p^\infty])/K(p^\infty))$ 
under both $\kappa_1$ and $\kappa_2$ are contained in $\mu_{p-1}$, the characters 
$\langle\kappa_1\rangle$ and $\langle\kappa_2\rangle$ are trivial on $\Delta$. 
They therefore descend to well-defined characters 
\[ 
\langle\kappa_1\rangle,\ \langle\kappa_2\rangle \colon \mathcal{G} \to 1+p\mathbb{Z}_p. 
\]

\begin{definition}[{\cite[II.4.17]{dS87}}]\label{def:two-variable_p-adic_L}
Let $\chi$ be a finite-order character of $\mathcal G$.
We define the two-variable $p$-adic $L$-function by
\[
L_p(s_1,s_2,\chi) = \int_{\mathcal G}\chi^{-1}\langle\kappa_1\rangle^{s_1}\langle\kappa_2\rangle^{s_2}(\sigma)\,d\mu_{\mathrm{dS}}^*(\sigma) \qquad (s_1, s_2 \in \mathbb{Z}_p). 
\]
\end{definition}

Let $\mathfrak{a}$ be a nontrivial ideal of $O_K$ prime to $p$. 
For integers $m, n \ge 1$, Lemma~\ref{lem: algebraicity of theta}~(ii) allows us to define the principal unit
\begin{equation}\label{def:elliptic_unit}
    \theta_{m,n}(\mathfrak{a}) := \left\langle \theta_{\mathfrak{a}}\left(\Omega; \frp^n\bar{\frp}^mL\right)\right\rangle \in U_{m,n}^\E.
\end{equation}
These units form a norm-compatible system $\boldsymbol{\theta}(\mathfrak{a}) = (\theta_{m,n}(\mathfrak{a}))_{m,n} \in \mathbf{U}_\infty^\E$. 
For integers $i,j$ satisfying $i\equiv j\pmod{w_K}$, we also define 
a finite-order character $\omega^{(i,j)}\colon \mathcal{G}\to \mathbb{Z}_p^\times$
by $\omega^{(i,j)}(\sigma) = \omega\bigl(\kappa_1^i\kappa_2^j(\sigma)\bigr)$.
The following proposition compares the $p$-adic measure $\mu_{\boldsymbol{\theta}(\mathfrak{a})} \in \Lambda(G, \widehat{\mathcal{O}})$ with the two-variable $p$-adic $L$-function defined above.

\begin{proposition}\label{prop:L-value}
    For each pair of integers $(i,j)$ satisfying
    $i\equiv j\pmod{w_K}$, we have
    \[
    \int_G \kappa_1^i\kappa_2^j(\sigma) \,d\mu_{\boldsymbol{\theta}(\mathfrak{a})}^*(\sigma)
    = 12\#\Delta\,
    \bigl(\kappa_1^i\kappa_2^j(\sigma_{\mathfrak{a}}) -\operatorname{N}\mathfrak{a}\bigr)
    L_p(i,j,\omega^{-(i,j)}), 
    \]
    where $\sigma_{\mathfrak{a}} = \left(\mathfrak{a}, K(p^\infty)/K\right) \in \mathcal{G}$ denotes the Artin symbol of $\mathfrak{a}$.
\end{proposition}

\begin{proof}
    Let $\mathbf{U}_\infty':=\varprojlim_{m,n}U_{m,n}'$, where, for $m,n\ge 1$, 
    $U_{m,n}'$ denotes the group of principal units of 
    $K(\frp^n\bar{\frp}^m)\otimes_K K_\frp$. 
    As in \eqref{eq:two-variable_measure}, we obtain a homomorphism 
    $i_\infty'\colon \mathbf{U}_\infty' \to\Lambda(\mathcal G,\widehat{\mathcal O})$. 
    The compatibility of the norm map $\nr\colon\mathbf{U}_\infty^\E \to \mathbf{U}_\infty'$ 
    with the push-forward $\res_*: \Lambda(G, \widehat{\mathcal{O}}) \to 
    \Lambda(\mathcal{G}, \widehat{\mathcal{O}})$ induced by $\res:G 
    \twoheadrightarrow \mathcal{G}$ gives $\res_*\circ i_\infty=i_\infty'\circ\nr$.
    By Lemma~\ref{lem: algebraicity of theta}~(ii),
    $\theta_{m,n}(\mathfrak a)\in U_{m,n}'$, and hence
    $\nr\bigl(\boldsymbol{\theta}(\mathfrak a)\bigr)
    = \boldsymbol{\theta}(\mathfrak a)^{\#\Delta}$.
    Therefore, the proof of \cite[II.4.14]{dS87} gives
    \[
    \res_*\mu_{\boldsymbol{\theta}(\mathfrak a)}
    = 12\#\Delta\,\bigl(\sigma_{\mathfrak{a}}-\operatorname{N}\mathfrak a\bigr)
    \mu_{\mathrm{dS}}.
    \]
    Since $i\equiv j\pmod{w_K}$, the character
    $\kappa_1^i\kappa_2^j$ descends to $\mathcal G$, and the desired
    equality follows from the above identity of measures.
\end{proof}

\section{Coleman-theoretic comparison}
In this section, we prove our main formula. 

\subsection{Congruences for Coleman power series and logarithmic derivatives}
First, we investigate the relation between the Coleman power series associated with $\widehat{\mathbb{G}}_m$ and those associated with $\widehat{E}$.
\begin{lemma}\label{lem: shift formula}
    Suppose that $f,g \in \widehat{\mathcal{O}}[[S]]$ satisfy
    $f(S)=\left((1+S)^{p^m}-1\right)g(S)$ for some integer $m\ge 0$. 
    Then, for any continuous function $\psi:\mathbb{Z}_p\to \widehat{\mathcal{O}}$, one has
    \[
    \int_{\mathbb{Z}_p} \psi(x)\,d\mu_f(x)
    =
    \int_{\mathbb{Z}_p} \bigl(\psi(x+p^m)-\psi(x)\bigr)\,d\mu_g(x),
    \]
    where $\mu_f,\mu_g\in \Lambda(\mathbb{Z}_p,\widehat{\mathcal{O}})$ are the measures corresponding to $f$ and $g$ via \eqref{eq:amice_transform}.
\end{lemma}
\begin{proof}
    Let $\delta_a$ be the Dirac measure at $a\in\mathbb{Z}_p$. 
    By the Amice transform,
    \[
    (1 + S)^{p^m} - 1 = \int_{\mathbb{Z}_p} (1 + S)^x\,d(\delta_{p^m} - \delta_0)(x). 
    \]
    This means that $\mu_{\left((1 + S)^{p^m} - 1\right)} = \delta_{p^m} - \delta_0$, 
    and hence $\mu_f = (\delta_{p^m} - \delta_0) \ast \mu_g$, so that
    \begin{align*}
        \int_{\mathbb{Z}_p} \psi(x)\,d\mu_f(x) &= \iint_{\mathbb{Z}_p\times\mathbb{Z}_p} \psi(x + y)\,d(\delta_{p^m} - \delta_0)(y)\,d\mu_g(x) \\
        &= \int_{\mathbb{Z}_p} (\psi(x + p^m) - \psi(x))\,d\mu_g(x). 
    \end{align*}
    This yields the desired equality. 
\end{proof}

\begin{proposition}\label{prop: relation of log derivative}
    Let $\boldsymbol{\epsilon}=(\boldsymbol{\epsilon}_m)_m = 
    (\epsilon_{m,n})_{m,n}\in\mathbf{U}_\infty^{\cyc}$ and 
    $\boldsymbol{\beta}=(\boldsymbol{\beta}_m)_m =
    (\beta_{m,n})_{m,n}\in\mathbf{U}_\infty^\E$ be norm-compatible 
    systems corresponding to each other via the isomorphism 
    $\mathbf{U}_\infty^{\cyc} \cong \mathbf{U}_\infty^\E$, 
    so that $\epsilon_{m,n}=\beta_{m,n}$ whenever $n\leq m$.
    Then the following assertions hold.
    \begin{enumerate}
        \item  For each $m\geq 1$, 
        \[
        f_{\boldsymbol{\epsilon}_m}(S) \equiv 
        (g_{\boldsymbol{\beta}_m}\circ\eta)(S) \pmod{(1+S)^{p^m}-1}.
        \]
        In particular, $f_{\boldsymbol{\epsilon}_m}(0) = 
        g_{\boldsymbol{\beta}_m}(0)$.

        \item For every $m\ge 1$ and $k \in \mathbb{Z}$, we have
        \[
        D^k\intlog(f_{\boldsymbol{\epsilon}_m})(0) \equiv D^k\intlog(g_{\boldsymbol{\beta}_m} \circ \eta)(0) \pmod{p^m}.
        \]
    \end{enumerate}
\end{proposition}

\begin{proof}
    (i) Fix $m\geq 1$. 
    By the interpolation properties and our assumption, we have
    \[
    f_{\boldsymbol{\epsilon}_m}^{\phi^{-n}}(\zeta_{p^n}-1)
    =
    (g_{\boldsymbol{\beta}_m}\circ\eta)^{\phi^{-n}}
    (\zeta_{p^n}-1)
    \qquad (1\leq n\leq m).
    \]
    Applying $\phi^n$ yields $f_{\boldsymbol{\epsilon}_m}(\xi_n-1)
    =(g_{\boldsymbol{\beta}_m}\circ\eta)(\xi_n-1)$, where $\xi_n:=\phi^n(\zeta_{p^n})$.
    Let $F_\infty:=\bigcup_{m'\geq 0}F_{m'}$. 
    Since $\xi_n$ is a primitive $p^n$-th root of unity and
    $\gal(F_\infty(\mu_{p^n})/F_\infty)$ acts trivially on $\widehat{\mathcal{O}}$, 
    taking Galois conjugates of these equalities for $1\leq n\leq m$ gives
    \[
    f_{\boldsymbol{\epsilon}_m}(\zeta-1) =
    (g_{\boldsymbol{\beta}_m}\circ\eta)(\zeta-1)
    \]
    for every $\zeta\in\mu_{p^m}\setminus\{1\}$. Weierstrass division
    \cite[Proposition~7.2]{Was97} therefore gives
    \begin{equation}\label{eq:preliminary_congruence}
    f_{\boldsymbol{\epsilon}_m}(S)
    \equiv
    (g_{\boldsymbol{\beta}_m}\circ\eta)(S)
    \pmod{((1+S)^{p^m}-1)/S}.
    \end{equation}
    Evaluating at $S=0$ and using $\eta(0)=0$, we obtain
    \[
    f_{\boldsymbol{\epsilon}_m}(0)
    \equiv
    g_{\boldsymbol{\beta}_m}(0)
    \pmod{p^m}.
    \]
    For $m'\geq m$, let
    $\nr_{m',m}\colon\mathcal O_{m'}\to\mathcal O_m$ denote the norm map.
    Arguing as in \cite[Lemma~2]{Yag82}, we have
    $\nr_{m',m}\bigl(f_{\boldsymbol{\epsilon}_{m'}}(0)\bigr)
    =f_{\boldsymbol{\epsilon}_m}(0)$ and
    $\nr_{m',m}\bigl(g_{\boldsymbol{\beta}_{m'}}(0)\bigr)
    =g_{\boldsymbol{\beta}_m}(0)$.
    Applying the preceding argument at level $m'$ and then the norm map
    $\nr_{m',m}$, which preserves congruences modulo $p^{m'}$, we obtain
    \[
    f_{\boldsymbol{\epsilon}_m}(0)
    \equiv
    g_{\boldsymbol{\beta}_m}(0)
    \pmod{p^{m'}}
    \]
    for every $m'\geq m$. Since $\mathcal{O}_m$ is $p$-adically separated,
    it follows that
    $f_{\boldsymbol{\epsilon}_m}(0)=g_{\boldsymbol{\beta}_m}(0)$.
    Combining this equality with \eqref{eq:preliminary_congruence} and using
    the fact that $\widehat{\mathcal{O}}$ is $p$-torsion-free, we obtain
    \[
    f_{\boldsymbol{\epsilon}_m}(S)
    \equiv
    (g_{\boldsymbol{\beta}_m}\circ\eta)(S)
    \pmod{(1+S)^{p^m}-1}.
    \]

    \noindent
    (ii) Let $h(S) := (f_{\boldsymbol{\epsilon}_m}/(g_{\boldsymbol{\beta}_m}\circ\eta))(S) \in \widehat{\mathcal{O}}[[S]]^\times$.
    By Lemma~\ref{lem: fundamental properties of Coleman power ser.} (ii) and 
    Lemma~\ref{lem: norm relation for elliptic Coleman series} (ii), 
    we have $\cnr(h) = h^\phi$. 
    Hence, $\mu_{\intlog(h)}$ is supported on $\mathbb{Z}_p^\times$ by Lemma~\ref{lem: fundamental properties of Coleman power ser.}~(iii).
    Since $\intlog(h) = \intlog(f_{\boldsymbol{\epsilon}_m}) -
    \intlog(g_{\boldsymbol{\beta}_m}\circ\eta)$, 
    it suffices to prove $D^k\intlog(h)(0) \equiv 0 \pmod{p^m}$. 
    By (i) and the fact that $\intlog$ descends to $\widehat{\mathcal{O}}[[S]] \big/ \bigl((1+S)^{p^m}-1\bigr)$, we may write 
    $\intlog(h)(S) = \left((1 + S)^{p^m} - 1\right)Q(S)$ for some $Q(S) \in \widehat{\mathcal{O}}[[S]]$.
    For $k\in\mathbb{Z}$, let
    \[
    \psi_k(x)=
    \begin{cases}
    x^k & \text{if } x\in\mathbb{Z}_p^\times,\\
    0 & \text{if } x\in p\mathbb{Z}_p.
    \end{cases}
    \]
    Since $\mu_{\intlog(h)}$ is supported on $\mathbb{Z}_p^\times$,
    Lemma~\ref{lem: shift formula} gives
    \[
    D^k\intlog(h)(0)
    = \int_{\mathbb{Z}_p}\psi_k(x)\,d\mu_{\intlog(h)}(x) 
    = \int_{\mathbb{Z}_p} \bigl(\psi_k(x+p^m)-\psi_k(x)\bigr) \,d\mu_Q(x).
    \]
    If $x\in p\mathbb{Z}_p$, the integrand is zero, while if
    $x\in\mathbb{Z}_p^\times$, then
    \[
    \psi_k(x+p^m)-\psi_k(x)=
    x^k\left(\left(1+\frac{p^m}{x}\right)^k-1\right)
    \in p^m\mathbb{Z}_p.
    \]
    Since $\mu_Q$ is an $\widehat{\mathcal{O}}$-valued measure, it follows that
    \[
    D^k\intlog(h)(0)\equiv 0 \pmod{p^m}.
    \]
    This proves the desired congruence.
\end{proof}

For later use, we give the analogous congruence for the ordinary logarithm.

\begin{lemma}\label{lem: relation of ordinary log derivative}
    Under the same assumptions as in 
    Proposition~\ref{prop: relation of log derivative}, 
    for every $m\ge 1$ and $k\ge 1$, we have
    \[
    D^k\log(f_{\boldsymbol{\epsilon}_m})(0)
    \equiv
    \delta_k(\boldsymbol{\beta}_m)
    \pmod{p^m}.
    \]
\end{lemma}

\begin{proof}
    By Proposition~\ref{prop: relation of log derivative}~(i), we may write
    \[
    h(S):=\frac{f_{\boldsymbol{\epsilon}_m}(S)}
    {(g_{\boldsymbol{\beta}_m}\circ\eta)(S)}
    =
    1+\left((1+S)^{p^m}-1\right)Q(S)
    \]
    for some $Q(S)\in\widehat{\mathcal{O}}[[S]]$.
    Hence $D^r(h-1)(0)\equiv 0\pmod{p^m}$ for $r \ge 1$.
    Since $D^k\log h(0)$ is a polynomial with zero constant term in
    $D(h-1)(0),\ldots,D^k(h-1)(0)$, we obtain
    \[
    D^k\log h(0)\equiv 0\pmod{p^m}.
    \]
    This proves the desired congruence.
\end{proof}

\subsection{Proof of the main formula}

We recall that $(\cdot,\cdot)_{m,n}\colon \mathcal{F}_m(\mu_{p^n})^\times \times \mathcal{F}_m(\mu_{p^n})^\times \to \mu_{p^n}$ denotes the Hilbert norm residue symbol defined in Definition~\ref{def:Hilbert_symbol}.

\begin{theorem}\label{thm:congruence}
    Let $\boldsymbol{\beta} = (\boldsymbol{\beta}_m)_m = (\beta_{m,n})_{m,n} \in \mathbf{U}_\infty^\E$. 
    Then, for each integer $m\ge1$ and each pair of positive integers $(i,j)$,
    the following holds:
    \[
    \left(\prod_{\sigma\in\gal(K(E[p^m])/K)} \sigma\left(\theta_{m,m}(\mathfrak{a})\right)^{\kappa_1^{i-1}\kappa_2^{j-1}(\sigma)}, \beta_{m,m} \right)_{m,m} 
    = \zeta_{p^m}^{A}, 
    \]
    where 
    \begin{align*}
        A \equiv 
        &(-1)^{i-1}(1-\pi^{j-1}\bar{\pi}^{i-1}) \sum_{\tau\in \gal(F_m/K)} \kappa_2^{j-1}(\tau)  D^{1-i}\intlog(g_{\widetilde{\tau}(\boldsymbol{\theta}_m(\mathfrak{a}))}\circ\eta)^*(0) \\
        &\phantom{(-1)^{i-1}(1-\pi^{i-1}\bar{\pi}^{j-1})} \cdot \sum_{\sigma\in \gal(F_m/K)} \kappa_2^{-j}(\sigma)\delta_i(\widetilde{\sigma}(\boldsymbol{\beta}_m))^* \pmod{p^m}.
    \end{align*}
    Here, $\widetilde{\sigma}, \widetilde{\tau} \in G$ are lifts of $\sigma$ and 
    $\tau$, respectively, whose actions on $E[\frp^\infty]$ are trivial, and the 
    right-hand side of the above congruence is independent of the choice of these 
    lifts.
\end{theorem}

\begin{proof}
    First, we rewrite the first entry on the left-hand side of the asserted equation.
    Let $\Gamma_m,R_m\subset G$ be sets of lifts of
    $\gal(K(\mu_{p^m})/K)$ and $\gal(F_m/K)$, respectively, such that
    the elements of $\Gamma_m$ act trivially on $F_m$, while those of $R_m$
    act trivially on $K(\mu_{p^m})$.
    Then we have a bijection
    \[
    \Gamma_m\times R_m \xrightarrow{\sim}\gal(K(E[p^m])/K),\; (\tau_1,\tau_2) \mapsto \left.\tau_1\tau_2\right|_{K(E[p^m])}.
    \]
    For $\sigma\in\gal(K(E[p^m])/K)$, let $(\tau_1,\tau_2)\in\Gamma_m\times R_m$ 
    be the unique pair corresponding to $\sigma$.
    The relation $\kappa_1\kappa_2=\chi_{\cyc}$ then gives
    \begin{equation}\label{eq:character_relation}
        \kappa_1^{i-1}\kappa_2^{j-1}(\sigma) = \chi_{\cyc}^{i-1}(\sigma)\kappa_2^{j-i}(\sigma)
        \equiv \chi_{\cyc}^{i-1}(\tau_1)\kappa_2^{j-i}(\tau_2) \pmod{p^m}.
    \end{equation}
    Let $(\boldsymbol{e}_m(\mathfrak{a}))_m = (e_{m,n}(\mathfrak{a}))_{m,n}
    \in \mathbf{U}_\infty^\cyc$ be the norm-compatible system corresponding
    under the isomorphism $\mathbf{U}_\infty^\E \cong \mathbf{U}_\infty^\cyc$
    to $\boldsymbol{\theta}(\mathfrak{a})\in\mathbf{U}_\infty^\E$
    defined by \eqref{def:elliptic_unit},
    so that $e_{m,n}(\mathfrak{a}) = \theta_{m,n}(\mathfrak{a})$ for $n\le m$.
    By \cite[II.4.9, Proposition (ii)]{dS87} and Proposition~\ref{prop: CPS a.w. ell} (i), 
    we have $g_{\boldsymbol{\theta}_m(\mathfrak{a})}(0)=1$.
    Hence, Proposition~\ref{prop: relation of log derivative} (i) and
    Lemma~\ref{lem: fundamental properties of Coleman power ser.} (ii) give
    \begin{equation}\label{eq:constant}
        f_{\boldsymbol{e}_m(\mathfrak{a})}(S) \in (1+S\mathcal{O}_m[[S]])^{\cnr=\phi}.
    \end{equation}
    For each $k \ge 0$, we also define $\omega_m(k) := \sum_{\sigma\in\Gamma_m} 
    \chi_\cyc^k(\sigma)\sigma \in \mathbb{Z}_p[G]$, which is the image of 
    the basic element under the ``weight accelerator'' homomorphism 
    (see the discussion preceding \cite[Proposition 4.2]{NSW17a}). 
    Then, we have
    \begin{align*}
        \left(f_{\boldsymbol{e}_m(\mathfrak{a})}^{\phi^{-m}}\right)^{\omega_m(i-1)}(\zeta_{p^m}-1) 
        := f_{\boldsymbol{e}_m(\mathfrak{a})}^{\phi^{-m}}(S)^{\omega_m(i-1)}\big|_{S=\zeta_{p^m}-1} 
        \phantom{:}= \prod_{\tau\in\Gamma_m} \tau\left(e_{m,m}(\mathfrak{a})\right)^{\chi_{\mathrm{cyc}}^{i-1}(\tau)}. 
    \end{align*}
    Therefore, by this identity and \eqref{eq:character_relation}, we have
    \begin{align*}
        \prod_{\sigma \in \mathrm{Gal}(K(E[p^m])/K)} \sigma\left(\theta_{m,m}(\mathfrak{a})\right)^{\kappa_1^{i-1}\kappa_2^{j-1}(\sigma)} 
        &\equiv \prod_{\tau_2\in R_m} \tau_2 \left( \prod_{\tau_1\in\Gamma_m} \tau_1\left(e_{m,m}(\mathfrak{a})\right)^{\chi_{\mathrm{cyc}}^{i-1}(\tau_1)} \right)^{\kappa_2^{j-i}(\tau_2)} \\
        &\equiv \prod_{\tau\in \gal(F_m/K)} \left(\left(f_{\boldsymbol{e}_m(\mathfrak{a})}^{\phi^{-m}\tau}\right)^{\omega_m(i-1)}(\zeta_{p^m}-1) \right)^{\kappa_2^{j-i}(\tau)}, 
    \end{align*}
    modulo $\mathcal{F}_m(\mu_{p^m})^{\times p^m}$. 
    Now, we take the norm-compatible system $(\boldsymbol{\epsilon}_m)_m=(\epsilon_{m,n})_{m,n} \in \mathbf{U}_\infty^\cyc$ corresponding to $\boldsymbol{\beta} \in \mathbf{U}_\infty^\E$, 
    so that $\epsilon_{m,n} = \beta_{m,n}$ for $n\le m$, and let
    $h_{\boldsymbol{\epsilon}_m}(S)$ denote the Coleman power series associated
    with $\boldsymbol{\epsilon}_m$.
    By \eqref{eq:constant}, we may apply Theorem~\ref{thm: Coleman} to obtain
    \[
    \left(
    \left(f_{\boldsymbol{e}_m(\mathfrak{a})}^{\phi^{-m}\tau}\right)^{\omega_m(i-1)}
    (\zeta_{p^m}-1), \epsilon_{m,m}
    \right)_{m,m}
    =
    \zeta_{p^m}^{\alpha(\tau)},
    \]
    where, by \cite[Lemma (4.6)]{Col89}, the exponent $\alpha(\tau)$ is given by
    \begin{align*}
        \alpha(\tau) &\equiv \tr_{\mathcal{F}_m/K_{\mathfrak{p}}}\left(\int_m \omega_m(i-1)\intlog\left(f_{\boldsymbol{e}_m(\mathfrak{a})}^{\phi^{-m}\tau}\right)(S) \cdot D\intlog\left(h_{\boldsymbol{\epsilon}_m}^{\phi^{-m}}\right)(S)\right) \\
        &\equiv \tr_{\mathcal{F}_m/K_{\mathfrak{p}}}\left((-1)^{i-1}D^{1-i}\intlog\left(f_{\boldsymbol{e}_m(\mathfrak{a})}^{\phi^{-m}\tau}\right)(0)\cdot D^i\intlog\left(h_{\boldsymbol{\epsilon}_m}^{\phi^{-m}}\right)(0)\right) \\
        &\equiv (-1)^{i-1}\tr_{\mathcal{F}_m/K_{\mathfrak{p}}}\left(D^{1-i}\intlog(f_{\boldsymbol{e}_m(\mathfrak{a})}^\tau)(0)\cdot (1 - p^{i-1}\phi)D^i(\log h_{\boldsymbol{\epsilon}_m})(0)\right)\pmod{p^m}. 
    \end{align*}
    Here, in the last congruence, we use that $\phi$ commutes with $\intlog$ and $D$. 
    From the bilinearity of the pairing, we obtain
    \[
    \left(\prod_{\tau \in \gal(F_m/K)} \left(\left(f_{\boldsymbol{e}_m(\mathfrak{a})}^{\phi^{-m}\tau}\right)^{\omega_m(i-1)}(\zeta_{p^m}-1)\right)^{\kappa_2^{j-i}(\tau)}, \epsilon_{m,m}\right)_{m,m} 
    = \zeta_{p^m}^{\sum_{\tau \in \gal(F_m/K)} \kappa_2^{j-i}(\tau) \cdot \alpha(\tau)}.
    \]
    Thus, the exponent on the right-hand side is precisely the desired quantity $A$.
    As in the proof of Lemma~\ref{lem: approximation}, we have
    \begin{align*}
        A &\equiv (-1)^{i-1}\tr_{\mathcal{F}_m/K_{\mathfrak{p}}}\left(\sum_{\tau\in \gal(F_m/K)} \kappa_2^{j-i}(\tau) D^{1-i}\intlog\left(f_{\boldsymbol{e}_m(\mathfrak{a})}^{\tau}\right)(0)\cdot (1 - p^{i-1}\phi)D^i(\log h_{\boldsymbol{\epsilon}_m})(0)\right) \\
        &\equiv (-1)^{i-1}(1-p^{i-1}\kappa_2^{j-i}(\phi)) \tr_{\mathcal{F}_m/K_\frp}\left(\sum_{\tau\in \gal(F_m/K)}\kappa_2^{j-i}(\tau)D^{1-i}\intlog\left(f_{\boldsymbol{e}_m(\mathfrak{a})}^{\tau}\right)(0)\cdot D^i(\log h_{\boldsymbol{\epsilon}_m})(0)\right) \\
        &\equiv (-1)^{i-1}(1 - \pi^{j-1}\bar{\pi}^{i-1}) \sum_{\sigma, \tau\in \gal(F_m/K)} \kappa_2^{j-i}(\tau) D^{1-i}\intlog\left(f_{\boldsymbol{e}_m(\mathfrak{a})}^{\sigma\tau}\right)^*(0) \cdot D^i(\log h_{\boldsymbol{\epsilon}_m}^\sigma)^*(0) \pmod{p^m}.
    \end{align*}
    Here, in the last congruence, we use \cite[Chap.~II, Corollary~8.4]{Neu99}. 
    Making the change of variables $\tau'=\sigma\tau$ in the double sum, we obtain
    \begin{align*}
        A &\equiv (-1)^{i-1}(1-\pi^{j-1}\bar{\pi}^{i-1}) \left(\sum_{\tau'\in \gal(F_m/K)}\kappa_2^{j-i}(\tau')  D^{1-i}\intlog\left(f_{\boldsymbol{e}_m(\mathfrak{a})}^{\tau'}\right)^*(0)\right) \\
        &\phantom{= (-1)^{i-1}(1-\pi^{i-1}\bar{\pi}^{j-1}) \quad} \cdot \left(\sum_{\sigma\in \gal(F_m/K)} \kappa_2^{i-j}(\sigma)D^i(\log h_{\boldsymbol{\epsilon}_m}^\sigma)^*(0)\right) \pmod{p^m}.
    \end{align*}
    For each $\tau \in \gal(F_m/K)$, take a lift $\widetilde{\tau} \in G$ acting trivially on $E[\frp^\infty]$. 
    Then, Lemma~\ref{lem: fundamental properties of Coleman power ser.} (i) and Proposition~\ref{prop: relation of log derivative} (ii) yield 
    \begin{align*}
        D^{1-i}\intlog(f_{\boldsymbol{e}_m(\mathfrak{a})}^\tau)^*(0) &\equiv \kappa_2^{i-1}(\tau) D^{1-i}\intlog(f_{\widetilde{\tau}(\boldsymbol{e}_m(\mathfrak{a}))})^*(0) \\
        &\equiv \kappa_2^{i-1}(\tau) D^{1-i}\intlog(g_{\widetilde{\tau}(\boldsymbol{\theta}_m(\mathfrak{a}))}\circ\eta)^*(0) \pmod{p^m}. 
    \end{align*}
    Similarly, by Lemma~\ref{lem: relation of ordinary log derivative}, we have
    \begin{align*}
        D^i(\log h_{\boldsymbol{\epsilon}_m}^\sigma)^*(0) \equiv \kappa_2^{-i}(\sigma) D^i(\log h_{\widetilde{\sigma}(\boldsymbol{\epsilon}_m)})^*(0) 
        \equiv \kappa_2^{-i}(\sigma) \delta_i(\widetilde{\sigma}(\boldsymbol{\beta}_m))^* \pmod{p^m}. 
    \end{align*}
    Hence, we obtain the desired congruence.
\end{proof}

By pulling back the elliptic Soul'{e} character along the composite
\[
\prod_{v\mid\frp} G_{K(E[p^\infty])v}^{\ab}
\xrightarrow{\lambda_\frp}
G_{K(E[p^\infty])}^{\ab}
\longrightarrow
G_{K(p^\infty)}^{\ab},
\]
where $\lambda_\frp$ is the homomorphism in \eqref{eq:embedding_to_global},
we regard it as a character on
$\prod_{v\mid\frp} G_{K(E[p^\infty])v}^{\ab}$.
Recall also that the reciprocity map \eqref{eq:reciprocity-map} 
gives an embedding $\rec_\infty\colon \mathbf{U}_\infty^\E
\hookrightarrow \prod_{v\mid\frp} G_{K(E[p^\infty])_v}^{\ab}$.

\begin{corollary}[{Theorem~\ref{thm:intro_CM_elliptic_Coleman_Ihara}}]
\label{cor:CM_elliptic_Coleman_Ihara}
    For each pair of positive integers $(i,j) \neq (1,1)$ such that
    $i\equiv j\pmod{w_K}$, the following identity holds:
    \[
    \begin{aligned}
        \frac{\chi_{(i,j)}^\E(\rec_\infty(\boldsymbol{\beta}))}
        {1-\pi^{j-1}\bar{\pi}^{i-1}}
            &= 12w_K
            \left(
            1-\frac{1}{1-\pi^{i-1}\bar{\pi}^{j-1}}
            -\frac{1}{1-\pi^{j-1}\bar{\pi}^{i-1}}
            \right) \\
            &\phantom{={}}
            \cdot
            \frac{L_p(1-i,1-j,\omega^{(i,j)-1})}
            {\Omega_p^{j-i}}
            \,\phi_{(i,j)}^\cw(\boldsymbol{\beta}) \qquad (\boldsymbol{\beta}\in \mathbf{U}_\infty^\E).
    \end{aligned}
    \]
\end{corollary}

\begin{proof}
    For each $m\ge 1$, by Lemma~\ref{lem: approximation} and 
    Definition~\ref{def:coates-wiles}, the exponent $A$ in 
    Theorem~\ref{thm:congruence} can be written as follows:
    \[
    A \equiv (-1)^{i-1}(1-\pi^{j-1}\bar{\pi}^{i-1})\Omega_p^{i-j}\int_G \kappa_1^{1-i}\kappa_2^{1-j}(\sigma)\,d\mu_{\boldsymbol{\theta}(\mathfrak{a})}^* \cdot \phi_{(i,j)}^\cw(\boldsymbol{\beta}) \pmod{p^m}.
    \]
    On the other hand, by Lemma~\ref{lem: algebraicity of theta}~(ii),
    we have $\theta_{\mathfrak{a}}(\Omega/p^m;L) \in K(p^m)$.
    Since $i \equiv j \pmod{w_K}$ and the Hilbert symbol defined in
    Definition~\ref{def:Hilbert_symbol} is independent of the
    $k_m^\times$-components of its entries, Definition~\ref{def:elliptic_soule}
    gives
    \[
    A \equiv \#\Delta\, \chi_{(i,j),\mathfrak{a}}^\E \bigl(\rec_\infty(\boldsymbol{\beta})\bigr)
    \pmod{p^m}.
    \]
    Therefore, Proposition~\ref{prop: compare character} and Proposition~\ref{prop:L-value} 
    yield the desired equality.
\end{proof}

\begin{remark}
In \cite[Corollary~1.10]{IG26}, we obtained a CM elliptic analogue of the
Coleman--Ihara formula along the cyclotomic line. We now restrict the above
corollary to the cyclotomic line, i.e., to $i=j$, and verify that 
the resulting identity agrees with the formula obtained there. To this end, 
we first compare the choices of generators used in the two formulations.

Let $\chi_m^\E$ denote the elliptic Soul\'{e} character defined in
\cite{IG26}. Recall that there we used
$(-\omega_{\frp,n},\omega_{\bar{\frp},n})$ as a basis of
$E[p^n]\cong E[\frp^n]\oplus E[\bar{\frp}^n]$
and $\exp\left(2\boldsymbol{\pi}\boldsymbol{i}/p^n\right)
= \zeta_{p^n}^{-\sqrt{-d_K}}$
as a generator of $\mu_{p^n}$.
By comparing the Kummer properties characterizing the elliptic Soul\'e
characters under these choices with those under the choices made in
Definition~\ref{def:elliptic_soule}, we obtain the following identity of
characters on $G_{K(p^\infty)}$:
\begin{equation}\label{eq:comparison of elliptic soule characters}
    \chi_{(m,m)}^\E
    =
    (-1)^m\sqrt{-d_K}\,\chi_m^\E.
\end{equation}
The same change of generator also affects the normalization of the
Coates--Wiles homomorphisms. 
In \cite{IG26}, we used the one-variable Coates--Wiles homomorphism
$\phi_m^\cw$ of \cite{NSW17a}, defined on $\mathbf{U}_\infty$, the group of
norm-compatible systems of principal units in the cyclotomic tower.
The cyclotomic specialization $\phi_{(m,m)}^\cw$ defined in
Definition~\ref{def:coates-wiles} is compatible with $\phi_m^\cw$ via the
norm map $\mathbf{U}_\infty^\E\to \mathbf{U}_\infty$. 
Under this compatibility, the change of generator gives
\begin{equation}\label{eq:comparison of coates-wiles homomorphisms}
    \phi_{(m,m)}^\cw
    =
    (-\sqrt{-d_K})^m\phi_m^\cw.
\end{equation}

We next explain a correction to the comparison between the Gauss sum and
Katz's local factor underlying Gross's factorization formula
\cite[Theorem]{Gr80}.
Let $\chi_\infty$ be a finite-order complex-valued character of
$\gal(K(p^\infty)/\mathbb{Q})$ of conductor $p^n$.
On p.~88 of \cite{Gr80}, Gross defines the Gauss sum by
\[
g(\chi_\infty)
=
\frac{1}{p^n}
\sum_{a=1}^{p^n}
\chi_\infty(a)
\exp\left(\frac{2\boldsymbol{\pi}\boldsymbol{i}a}{p^n}\right).
\]
When $\chi_\infty$ is even, Gross claims in \cite[\S2.5, (3)]{Gr80} that
Katz's local factor
$\operatorname{Local}(\chi_\infty)
:=\operatorname{Local}(\chi,\Sigma,\delta)$, with
$\delta=\sqrt{-d_K}/2$, satisfies
$\operatorname{Local}((\chi_K)_\infty)=g(\chi_\infty^{-1})$,
where $(\chi_K)_\infty$ denotes the restriction of $\chi_\infty$ to
$\gal(K(p^\infty)/K)$. However, for the choice
$\delta=\sqrt{-d_K}/2$, the additive character used in the Fourier
transform defining $\operatorname{Local}(\chi_\infty)$ is the one
determined by $\zeta_{p^n}^{-1}$. Therefore, the correct relation is
\[
g(\chi_\infty)
=
\chi_\infty(\sqrt{-d_K})\operatorname{Local}(\chi_\infty).
\]
Thus, writing $L_p^{\mathrm{Katz}}(s,\chi)$ for Katz's $p$-adic
$L$-function defined in \cite[(2.4)]{Gr80}, we find that the left-hand
side of the equality in \cite[\S2.5, (3)]{Gr80},
\[
L_p^{\mathrm{Katz}}(0,\chi_K)
=
-\frac{1}{3p^n}g(\chi^{-1})
\sum_A\chi(a)\log_p F^+(a)_p,
\]
should be multiplied by $\chi(\sqrt{-d_K})$, so that \cite[(3.4)]{Gr80}
should read
\[
\chi(\sqrt{-d_K})L_p^{\mathrm{Katz}}(0,\chi_K)
=
L_p(0,\chi\varepsilon\omega)L_p(1,\chi^{-1}),
\]
where $\varepsilon$ is the Dirichlet character associated with $K$.
In the notation of \cite{Gr80}, let $d$ be the generator of the common
order-$2$ inertia subgroup $I_\ell$ in
$\gal(K(\mu_{p^\infty})/\mathbb{Q})/\langle c\rangle$ for primes
$\ell\mid d_K$, where $c$ denotes complex conjugation, and let
$\boldsymbol{N}=\chi_\cyc$ denote the cyclotomic character.
Then, with this correction, Gross's factorization formula takes the following
form:
\[
L_p^{\mathrm{Katz}}(s,\chi_K)
=
\begin{cases}
(\chi^{-1}\langle\boldsymbol{N}\rangle^s)
(\sigma_{\sqrt{-d_K}})
L_p(s,\chi\varepsilon\omega)
L_p(1-s,\chi^{-1}),
& \text{if } \chi(d)=1,\\[4pt]
(\chi\varepsilon\omega\langle\boldsymbol{N}\rangle^{1-s})
(\sigma_{\sqrt{-d_K}})
L_p(s,\chi\varepsilon\omega)
L_p(1-s,\chi^{-1}),
& \text{if } \chi(d)=-1.
\end{cases}
\]
If $s=m$ is even and $\chi=\omega^{-m}$, then
$\chi(d)=1$, and hence
\[
(\sqrt{-d_K})^{-m}L_p^{\mathrm{Katz}}(m,\omega^{-m})
=
L_p(m,\varepsilon\omega^{1-m})
L_p(1-m,\omega^m).
\]
On the other hand, if $s=m$ is odd and $\chi=\varepsilon\omega^{-m}$, 
then $\chi(d)=-1$, and we obtain
\[
(\sqrt{-d_K})^{m-1}L_p^{\mathrm{Katz}}(m,\omega^{-m})
=
L_p(m,\omega^{1-m})
L_p(1-m,\varepsilon\omega^m).
\]
Replacing Gross's factorization formula used in the proof of
\cite[Corollary~1.10]{IG26} with the above two formulas, and using
\eqref{eq:comparison of elliptic soule characters} and
\eqref{eq:comparison of coates-wiles homomorphisms}, we obtain, 
for both even and odd $m$, 
\[
\frac{\chi_{(m,m)}^\E(\rec_\infty(\boldsymbol{\beta}))}{1-p^{2m-2}}
=
3w_K L_p^{\mathrm{Katz}}(m,\omega^{-m})
\frac{\phi_{(m,m)}^\cw(\boldsymbol{\beta})}{p^{m-1}-1}
\qquad
(\boldsymbol{\beta}\in\mathbf{U}_\infty^\E).
\]

Finally, one checks that, for the choice $\delta=\sqrt{-d_K}/2$, 
Katz's $p$-adic $L$-function and the restriction to the cyclotomic line
of the two-variable $p$-adic $L$-function defined in
Definition~\ref{def:two-variable_p-adic_L} differ by a
factor of $4$ (see also \cite[Remarks~4.1 and~6.5]{BS22}):
\[
L_p^{\mathrm{Katz}}(s,\chi)
=
4L_p(s,s,\chi). 
\]
Hence, combining this relation with Katz's functional equation
\cite[(5.3.7)]{Ka78},
\begin{equation}\label{eq:functional_equation_Katz}
    L_p^{\mathrm{Katz}}(m,\omega^{-m})
    =
    L_p^{\mathrm{Katz}}(1-m,\omega^{m-1}),
\end{equation}
we verify that the corrected formula above agrees with the restriction of
Corollary~\ref{cor:CM_elliptic_Coleman_Ihara} to the cyclotomic line.
\end{remark}

\begin{remark}
The functional equation for the $p$-adic $L$-function is also stated in
\cite[II.6.4]{dS87}, but its proof requires a correction. 
Indeed, the argument on p.~95 asserts that, for a continuous $p$-adic
character $\varepsilon=\varepsilon_G\varepsilon_H$, the $p$-adic root
number $W^{p\text{-}\mathrm{adic}}(\varepsilon)$ depends only on the
finite component $\varepsilon_H$. In general, however,
$W^{p\text{-}\mathrm{adic}}(\varepsilon)$ also depends on $\varepsilon_G$.
After taking this dependence into account, the resulting functional equation
for a Gr\"ossencharacter $\varepsilon$ of infinity type $(k,j)$, in the
notation of \cite[II.6.4]{dS87}, is
\[
L_p(\varepsilon)
=
\bigl(d_K\operatorname{N}\mathfrak{f}\bigr)^{\frac{k+j+1}{2}}
W^{p\text{-}\mathrm{adic}}(\varepsilon)
\frac{\check{\varepsilon}(\sigma_{-\delta})}
     {\varepsilon(\sigma_{\delta})}
L_p(\check{\varepsilon}),
\]
where $\mathfrak{f}$ denotes the prime-to-$p$ part of the conductor of
$\varepsilon$.
In particular, when $\mathfrak{f}=(1)$, the product of the factors
multiplying $L_p(\check{\varepsilon})$ on the right-hand side is equal to
$1$. Hence,
\[
L_p(\varepsilon)=L_p(\check{\varepsilon}),
\]
which is consistent with Katz's functional equation
\eqref{eq:functional_equation_Katz}.
\end{remark}

{\bf Acknowledgements.}
The first author was supported by The Keio University Doctorate Student Grant-in-Aid Program from Ushioda Memorial Fund (Graduate school recommendation).
The second author was supported by JSPS KAKENHI Grant Number JP23KJ1882 and JP26K24484. 

\bibliographystyle{amsalpha}
\bibliography{reference}

\end{document}